\documentclass[12pt,twoside]{article}
\usepackage{amsmath,amssymb,amsthm,mathrsfs,color,times,textcomp,sectsty,verbatim}
\usepackage[colorlinks=true]{hyperref}
\hypersetup{urlcolor=blue, citecolor=red, linkcolor=blue}
 \def \no{\nonumber}

 \allowdisplaybreaks
\begin{document}
\renewcommand{\proofname}{\bf Proof}
\let\oldsection\section
\renewcommand\section{\setcounter{equation}{0}\oldsection}
\renewcommand\thesection{\arabic{section}}
\renewcommand\theequation{\thesection.\arabic{equation}}
\newtheorem{claim}{\indent Claim}[section]
\newtheorem{theorem}{\indent Theorem}[section]
\newtheorem{lemma}{\indent Lemma}[section]
\newtheorem{proposition}{\indent Proposition}[section]
\newtheorem{definition}{\indent Definition}[section]
\newtheorem{remark}{\indent Remark}[section]
\newtheorem{corollary}{\indent Corollary}[section]
\newtheorem{example}{\indent Example}[section]
\newtheorem{exercise}{\indent Exercise}[section]
\newtheorem{case}{\indent Case}

\newcommand{\nn}{\nonumber}

\title{\Large \bf
Remarks on GJMS operator of order six}

\author{Xuezhang  Chen\thanks{ X. Chen: xuezhangchen@nju.edu.cn; $^\dag$F. Hou: houfeimath@gmail.com.} ~~and Fei Hou$^\dag$\\
 \small
$^\ast$ $^\dag$Department of Mathematics \& IMS, Nanjing University, Nanjing
210093, P. R. China
}

\date{}
\maketitle
\begin{abstract}
We study analysis aspects of the sixth order GJMS operator $P_g^6$.  Under conformal normal coordinates around a point, the expansions of Green's function of $P_g^6$ with pole at this point are presented. As a starting point of the study of $P_g^6$, we manage to give some existence results of prescribed $Q$-curvature problem on Einstein manifolds. One among them is that for $n \geq 10$, let $(M^n,g)$ be a closed Einstein manifold of positive scalar curvature and $f$ a smooth positive function in $M$. If the Weyl tensor is nonzero at a maximum point of $f$ and $f$ satisfies a vanishing order condition at this maximum point, then there exists a conformal metric $\tilde g$ of $g$ such that its $Q$-curvature $Q_{\tilde g}^6$ equals $f$.

 {{\bf $\mathbf{2010}$ MSC:} Primary 53A30,58J05,53C21;~~ Secondary 35J35,35J08,35B50.}

\end{abstract}

\section{Introduction}
\indent \indent Recently, some remarkable developments have been achieved in the existence theory of positive constant $Q$-curvature problem associated to Paneitz-Branson operator. One key ingredient in such works is that a strong maximum principle for the fourth order Paneitz-Branson operator is discovered under a hypothesis on the positivity of some conformal invariants or $Q$-curvature of the background metric. The readers are referred to \cite{gursky_hang_lin,gur_mal,hy3,lx} and the references therein. This naturally stimulates us to study GJMS operator of order six and its associated $Q$-curvature problem, the analogue to the Yamabe problem and $Q$-curvature problem for Paneitz-Branson operator. Except for the aforementioned cases, due to the lack of a maximum principle for higher order elliptic equations in general, the existence theory of such problems needs to be developed. Until an analogue of Aubin's result in \cite{aubin} for the Yamabe problem is verified in Proposition \ref{prop:Y_6(M)_n>9} below, by adapting some ideas for Paneitz-Branson operator from \cite{er, dhl} we establish some existence results of prescribed $Q$-curvature problem on Einstein manifolds, in which case the sixth order GJMS operator is of constant coefficients.

The conformally covariant GJMS operators with principle part $(-\Delta_g)^k, k \in \mathbb{N}$ are discovered by Graham-Jenne-Mason-Sparling \cite{gjms}. In particular, the GJMS operator of order six and the associated $Q$-curvature are given as follows (cf. \cite{juhl,wunsch}): on manifolds $(M^n,g)$ of dimension $n\geq 3$ and $n \neq 4$, denote by $\sigma_k(A_g)$ the $k$-th elementary symmetric function of Schouten tensor $A_{ij}=\frac{1}{n-2}(R_{ij}-\frac{R_g}{2(n-1)}g_{ij})$. Denote by
 $$C_{ijk}=\nabla_k A_{ij}-\nabla_j A_{ik}, \quad B_{ij}=\Delta_g A_{ij}-\nabla^k \nabla_j A_{ik}-A^{kl}W_{kijl}=\nabla^k C_{ijk}-A^{kl}W_{kijl}$$
 the Cotton tensor and Bach tensor, respectively. Let
\begin{align*}
T_2=&(n-2) \sigma_1(A_g)g-8A_g=-\frac{8}{n-2}{\rm Ric}_g+\frac{n^2-4n+12}{2(n-1)(n-2)}R_gg;\\
T_4=&-\frac{3n^2-12n-4}{4}\sigma_1(A_g)^2 g+4(n-4)|A|_g^2 g+8(n-2)\sigma_1(A_g)A_g\\
&+(n-6)\Delta_g \sigma_1(A_g) g-48A_g^2-\frac{16}{n-4}B_g;\\
v_6=&-\frac{1}{8}\sigma_3(A_g)-\frac{1}{24(n-4)}\langle B,A\rangle_g.
\end{align*}
Then, the $Q$-curvature $Q_g^6$ is defined by
\begin{align}
Q_g^6=&-3! 2^6 v_6-\frac{n+2}{2}\Delta_g (\sigma_1(A_g)^2)+4\Delta_g |A|_g^2-8\delta(A_gd\sigma_1(A_g))+\Delta_g^2 \sigma_1(A_g)\no\\
&-\frac{n-6}{2}\sigma_1(A_g)\Delta_g\sigma_1(A_g)-4(n-6)\sigma_1(A_g)|A|_g^2+\frac{(n-6)(n+6)}{4}\sigma_1(A_g)^3\label{Q_g^6}
\end{align}
and the GJMS operator of sixth order $P_g^6$ is given by\footnote{The definition of $P_g^6$ differs from the formula (10.15) in \cite{juhl} by a minus sign.}
\begin{equation}\label{P_g}
-P_g^6=\Delta_g^3+\Delta_g \delta T_2 d+\delta T_2 d \Delta_g+\frac{n-2}{2}\Delta_g(\sigma_1(A_g)\Delta_g)+\delta T_4 d-\frac{n-6}{2}Q_g^6,
\end{equation}
where $-\delta d=\Delta_g$. The operator $P_g^6$ is conformally covariant in the sense that if $\tilde g=u^{\frac{4}{n-6}}g, 0<u\in C^\infty(M)$ with $n \geq 3$ and $n \neq 4,6$,
\begin{equation}\label{conf_invar_P_g}
u^{\frac{n+6}{n-6}}P_{\tilde g}^6 \varphi=P_g^6(u\varphi)
\end{equation}
and in dimension $6$,
$$P_{e^{2u}g}^6 \varphi=e^{-6u}P_g^6 \varphi$$
for all $\varphi \in C^\infty(M)$. When $(M,g)$ is Einstein, $P_g^6$ is of constant coefficients, explicitly,
\begin{align*}
Q_g^6=&\frac{n^4-20n^2+64}{32n^2(n-1)^3}R_g^3,\\
-P_g^6=&\Delta_g^3+\frac{-3n^2+6n+32}{4n(n-1)}R_g\Delta_g^2+\frac{3n^4-12n^3-52n^2+128n+192}{16n^2(n-1)^2}R_g^2\Delta_g-\frac{n-6}{2}Q_g^6.
\end{align*}
Obviously, when $n\geq 7$, $Q_g^6$ is a positive constant whenever the scalar curvature $R_g$ is positive. Through a direct computation, the GJMS operator $P_g^6$ has the following factorization:
 \begin{equation}\label{factorization_GJMS}
 P_g^6=\Big(-\Delta_g+\frac{(n-6)(n+4)}{4n(n-1)}R_g\Big)\Big(-\Delta_g+\frac{(n-4)(n+2)}{4n(n-1)}R_g\Big)\Big(-\Delta_g+\frac{n-2}{4(n-1)}R_g\Big).
 \end{equation}
 In general, as shown in \cite{fg} and \cite{gover}, on Einstein manifolds the GJMS operator of order $2k$ for all positive integers $k$ satisfies the above property as
 $$P_g^{2k}=\prod_{i=1}^k\Big(-\Delta_g+\frac{R_g}{4n(n-1)}(n+2i-2)(n-2i)\Big).$$
In particular, choose $M^n=S^n, g=g_{S^n}$, then
\begin{align*}
Q_{S^n}^6=&\tfrac{n(n^4-20 n^2+64)}{32},\\
P_{S^n}^6=&-\Delta_{S^n}^3-\tfrac{-3n^2+6n+32}{4}\Delta_{S^n}^2-\tfrac{3n^4-12n^3-52n^2+128n+192}{16}\Delta_{S^n}+\tfrac{n-6}{2}Q_{S^n}^6\\
=&\Big(-\Delta_{S^n}+\tfrac{(n-6)(n+4)}{4}\Big)\Big(-\Delta_{S^n}+\tfrac{(n-4)(n+2)}{4}\Big)\Big(-\Delta_{S^n}+\tfrac{n(n-2)}{4}\Big).
\end{align*}
From now on, we set $P_g=P_g^6$ and $Q_g=Q_g^6$ in the whole paper unless stated otherwise. Then, for any $\varphi \in H^3(M,g)$, we get
\begin{align*}
&\int_M \varphi P_g \varphi d\mu_g\\
=&\int_M\Big(|\nabla \Delta \varphi|_g^2-2T_2(\nabla \Delta \varphi,\nabla \varphi)-\frac{n-2}{2}\sigma_1(A)(\Delta_g \varphi)^2-T_4(\nabla \varphi,\nabla \varphi)+\frac{n-6}{2}Q_g \varphi^2\Big) d\mu_g.
\end{align*}

As a starting point of the study on the sixth order GJMS operator, we obtain some existence results of conformal metrics with positive $Q$-curvature candidates on closed Einstein manifolds under some additional natural assumptions.
\begin{theorem}\label{main_Thm}
Suppose $(M^n,g)$ is a closed Einstein manifold of dimension $n\geq 10$ and has positive scalar curvature. Let $f$ be a smooth positive function in $M$. Assume the Weyl tensor $W_g$ is nonzero at a maximum point $p$ of $f$ and $f$ satisfies the vanishing order condition at $p$:
\begin{equation}\label{vanishing_order_con}
\left \{\begin{array}{ll}
\Delta_g f(p)=0, &\hbox{~~if~~} n=10,\\
\nabla^k f(p)=0, k=2,3,4, &\hbox{~~if~~} n \geq 11.
\end{array}
\right.
\end{equation}
 Then there exists a smooth solution to the $Q$-curvature equation
\begin{equation*}
P_g u=fu^{\frac{n+6}{n-6}}, \quad u>0\hbox{~~in~~} M.
\end{equation*}
\end{theorem}

We remark that the condition \eqref{vanishing_order_con} imposed on the $Q$-curvature candidates $f$ is conformally invariant. The condition that $(M,g)$ is Einstein is only used to seek a \textit{\underline{positive}} solution.Theorem \ref{main_Thm} is a special case of a generalized Theorem \ref{generalized_main_Thm}.

This paper is organized as follows. In section \ref{sec2}, the expansions of Green's function for $P_g$ when $n \geq 7$ are presented under  conformal normal coordinates around a point. The technique used here is basically inspired by Lee-Parker \cite{lp}, see also \cite{hy3}. The complicate computations of the term $P_g(r^{6-n})$ are left to Appendix \ref{app_A}, where $r$ is the geodesic distant from this point. In section \ref{sec3}, we prove an analogue (cf. Proposition \ref{prop:Y_6(M)_n>9}) of Aubin's result for any closed manifold of dimension $n \geq 10$, which is not locally conformal flat. Based on this result, using the Moutain Pass Lemma we state in Theorem \ref{generalized_main_Thm} some results of prescribed $Q$-curvature problem associated to the sixth order GJMS operator on Einstein manifolds. Then our main Theorem \ref{main_Thm} directly follows from Theorem \ref{generalized_main_Thm}.

\noindent{\bf Acknowledgments:}
The first author is supported through NSFC (No.11201223), A Foundation for the Author of National Excellent Doctoral Dissertation of China (No.201417)  and Program for New Century Excellent Talents in University (NCET-13-0271). We thank the referee for careful reading and for helpful advice which improves the exposition.

\section{Expansion of Green's function of $\mathbf{P_g}$}\label{sec2}

Based on the survey paper by Lee-Parker \cite{lp} on the Yamabe problem, the method of deriving expansions of Green's function of $P_g$ is more or less standard except for careful computations on some lower order terms involved in $P_g$. One may also refer to \cite{hy3} for the Paneitz-Branson operator case. Green's functions of conformally covariant operators play an important role in the solvability of the constant curvature problems, for instance, the Yamabe problem (cf. \cite{lp} etc.) and the constant $Q$-curvature problem for Paneitz-Branson operator (cf.\cite{dhl,er,gursky_hang_lin,hy3}, etc.). Especially, F. Hang and P. Yang \cite{hy3} set up a dual variational method of the minimization for the Paneitz-Branson functional to seek a positive maximizer of the dual functional, such a scheme heavily relies on the positivity and expansion of its Green's function. We expect that the expansion of Green's function for $P_g^6$ will be useful to some possible future applications.

Throughout the whole paper, we use the following notation: $2^\sharp=\frac{2n}{n-6}, \omega_n={\rm vol}(S^n,g_{S^n})$ and when $n>6$, $c_n=\frac{1}{8(n-2)(n-4)(n-6)\omega_{n-1}}$. For $m \in \mathbb{Z}_+$, let
$$\mathcal{P}_m:=\{\hbox{homogeneous polynomials in~~}\mathbb{R}^n \hbox{~~of degree~~} m\}$$
and
$$\mathcal{H}_m:=\{\hbox{harmonic polynomials in~~}\mathbb{R}^n \hbox{~~of degree~~} m \}.$$
Moreover, $\mathcal{P}_m$ has the following decomposition (cf. \cite{stein}, P. 68-70)
$$\mathcal{P}_m=\bigoplus_{k=0}^{[\frac{m}{2}]}(r^{2k}\mathcal{H}_{m-2k}).$$

\begin{proposition}\label{prop_expan_G}
Assume $n>6$ and ${\rm ker} P_g=0$. Let $G_p(x)$ be the Green's function of sixth order GJMS operator at the pole $p\in M^n$ with the property that $P_g G_p=c_n \delta_p$ in the sense of distribution. Then, under the conformal normal coordinates around $p$ with conformal metric $g$, $G_p(x)$ has the following expansions:
\begin{enumerate}
\item[(a)] If $n$ is odd, then
$$G_p(x)=r^{6-n}(1+\sum_{k=1}^n \psi_k)+A+O(r),$$
where $A$ is a constant and $\psi_k \in \mathcal{P}_k$.
\item[(b)] If $n$ is even, then
\begin{eqnarray*}
G_p(x)&=&r^{6-n}(1+\sum_{k=1}^n \psi_k)+r^{6-n}\Big(\sum_{k=n-4}^n \varphi_k\Big) \log r+r^{6-n}\Big(\sum_{k=n-4}^n \varphi'_k\Big) \log^2 r\\
&&+r^{6-n}\Big(\sum_{k=n-2}^n \varphi''_k\Big) \log^3 r+\varphi'''_n \log^4 r+A+O(r),
\end{eqnarray*}
where $A$ is a constant and $\psi_k,\varphi_k,\varphi'_k,\varphi''_k,\varphi'''_k \in \mathcal{P}_k$.
\end{enumerate}
Moreover, we may restate some of the above results in another way.
\begin{enumerate}
\item[(c)] If $n=7,8,9$ or $M$ is conformally flat near $p$, then
$$G_p(x)=c_n r^{6-n}+A+O(r),$$
where $A$ is a constant.
\item[(d)] If $n=10$, then
$$G_p(x)=c_n r^{-4}+\frac{1}{17280}|W(p)|^2 \log r+O(1).$$
\item[(e)] If $n \geq 11$, then
$$G_p(x)=c_n r^{6-n}+\psi_4 r^{6-n}+O(r^{11-n}),$$
where $\psi_4 \in \mathcal{P}_4$ and
\begin{align*}
\psi_4(x)=&\frac{1}{135(n-2)}\Big[\sum_{k,l}(W_{iklj}(p)x^ix^j)^2-\frac{r^2}{n+4}\sum_{k,l,s}\big((W_{ikls}(p)+W_{ilks}(p))x^i\big)^2\\
                &+\frac{3}{2(n+4)(n+2)}|W(p)|^2 r^4\Big]\\
                &+\frac{3n-20}{270(n+4)(n-4)(n-8)}r^2\Big[\sum_{k,l,s}\big((W_{ikls}(p)+W_{ilks}(p))x^i\big)^2-\frac{3}{n}|W(p)|^2 r^2\Big]\\
                &-\frac{5n^2-66n+224}{120(n-8)(n-4)}r^2\Big[\sigma_1(A)_{,ij}(p)x^ix^j+\frac{|W(p)|^2}{12n(n-1)}r^2\Big]\\
                &+\frac{3n^4-16n^3-164n^2+400n+2432}{576(n+4)(n+2)n(n-1)}|W(p)|^2 r^4\Big\}.
\end{align*}
\end{enumerate}
\end{proposition}

Before starting to derive the expansion of Green's function of $P_g$, we first need to introduce some notation. For $\alpha \in \mathbb{R}$, set
\begin{align*}
       A_\alpha=&r^2\Delta_0+2\alpha r\partial_r+\alpha(\alpha+n-2),\\
       A_{\alpha,g}=&r^2\Delta_g+2\alpha r\partial_r+\alpha(\alpha+n-2),
\end{align*}
where $\Delta_0$ denotes the Euclidean Laplacian, and
\begin{eqnarray*}
       B_\alpha=\frac{\partial}{\partial\alpha}A_\alpha=2r\partial_r+2\alpha+n-2.
\end{eqnarray*}
For $k \in \mathbb{Z}_+$, a straightforward computation yields (also see \cite[Lemma 2.4]{hy3})
$$A_\alpha(\varphi \log^k r)=A_\alpha \varphi \log^k r+kB_\alpha \varphi \log^{k-1}r+k(k-1)\varphi \log^{k-2} r.$$
From this, for $\alpha,\beta,\gamma \in \mathbb{R}$ we get 
\begin{align}\label{A^3_varphi_logr}
&A_{\gamma}A_\beta A_\alpha (\varphi \log^k r)\no\\
=&A_{\gamma}A_\beta A_\alpha \varphi \log^k r\no\\
&+k(B_\gamma A_\beta A_\alpha+A_\gamma B_\beta A_\alpha+A_\gamma A_\beta B_\alpha)\varphi \log^{k-1} r\no\\
&+k(k-1)(A_\beta A_\alpha+B_\gamma B_\beta A_\alpha+B_\gamma A_\beta B_\alpha+A_\gamma B_\beta B_\alpha+A_\gamma A_\alpha+A_\gamma A_\beta)\varphi \log^{k-2}r\no\\
&+k(k-1)(k-2)(B_\beta A_\alpha+A_\beta B_\alpha+B_\gamma A_\alpha+B_\gamma B_\beta B_\alpha+B_\gamma A_\beta+A_\gamma B_\alpha+A_\gamma B_\beta)\varphi \log^{k-3}r\no\\
&+k(k-1)(k-2)(k-3)(A_\alpha+A_\beta+A_\gamma+B_\gamma B_\beta+B_\gamma B_\alpha+B_\beta B_\alpha)\varphi \log^{k-4} r\no\\
&+k(k-1)(k-2)(k-3)(k-4)(B_\alpha+B_\beta+B_\gamma)\varphi \log^{k-5}r\no\\
&+k(k-1)(k-2)(k-3)(k-4)(k-5)\varphi \log^{k-6}r.
\end{align}

A direct computation yields
\begin{align*}
       \Delta_0(r^\alpha\varphi)   &= r^{\alpha-2}A_\alpha\varphi, \\
       \Delta_0^2(r^\alpha\varphi) &= \Delta_0(r^{\alpha-2}A_\alpha\varphi)=r^{\alpha-4}A_{\alpha-2}A_\alpha\varphi, \\
       \Delta_0^3(r^\alpha\varphi) &= r^{\alpha-6}A_{\alpha-4}A_{\alpha-2}A_\alpha\varphi.
\end{align*}
In particular,
\begin{equation*}
       \Delta_0^3(r^{6-n}\varphi) = r^{-n}A_{2-n}A_{4-n}A_{6-n}\varphi.
\end{equation*}
Define
$$M_g:=\Delta_g \delta T_2 d+\delta T_2 d \Delta_g+\frac{n-2}{2}\Delta_g(\sigma_1(A)\Delta_g)+\delta T_4 d,$$
then rewrite \eqref{P_g} as
$-P_g=(\Delta_g)^3+M_g-\frac{n-6}{2}Q_g.$
Notice that
\begin{align*}
A_{\alpha,g}=&A_\alpha+r^2(\Delta_g-\Delta_0)=A_\alpha+r^2\partial_i((g^{ij}-\delta^{ij})\partial_j),\\
-P_g(r^\alpha \varphi)=&r^{\alpha-6}(A_{\alpha-4}A_{\alpha-2}A_\alpha \varphi+K_\alpha \varphi),
\end{align*}
where
\begin{align}
K_\alpha \varphi=&r^2(\Delta_g-\Delta_0)A_{\alpha-2}A_\alpha \varphi+A_{\alpha-4}(r^2(\Delta_g-\Delta_0))A_\alpha \varphi+A_{\alpha-4}A_{\alpha-2}(r^2(\Delta_g-\Delta_0)) \varphi\no\\
&+r^{6-\alpha}M_g(r^\alpha \varphi)-\frac{n-6}{2}r^6Q_g \varphi.\label{K_alpha}
\end{align}

We first state the expression of $P_g(r^{6-n})$ and leave complicate computations to Appendix \ref{app_A}.
\begin{lemma}\label{lem_P_g_r^{6-n}}
Under conformal normal coordinates around $p$ with metric $g$, there holds
\begin{align*}
&-P_g(r^{6-n}) \\
                 =&-c_n \delta_p+(n-6)r^{-n}\Big\{\tfrac{64(n-4)}{9}\Big[\sum_{k,l}(W_{iklj}(p)x^ix^j)^2-\tfrac{r^2}{n+4}\sum_{k,l,s}\big((W_{ikls}(p)+W_{ilks}(p))x^i\big)^2\\
                &+\tfrac{3}{2(n+4)(n+2)}|W(p)|^2 r^4\Big]+\tfrac{16(3n-20)}{9(n+4)}r^2\Big[\sum_{k,l,s}\big((W_{ikls}(p)+W_{ilks}(p))x^i\big)^2-\frac{3}{n}|W(p)|^2 r^2\Big]\\
                &-4(5n^2-66n+224)r^2\Big[\sigma_1(A)_{,ij}(p)x^ix^j+\frac{|W(p)|^2}{12n(n-1)}r^2\Big]\\
                &+\frac{3n^4-16n^3-164n^2+400n+2432}{3(n+4)(n+2)n(n-1)}|W(p)|^2 r^4\Big\}+O(r^{5-n}),
\end{align*}
where $W_{ijkl}$ is the Weyl tensor of metric $g$ and each term in square brackets on the right hand side of the identity is harmonic polynomial.

\end{lemma}
Consequently, we rewrite the above equation in Lemma \ref{lem_P_g_r^{6-n}} as
$$P_g(r^{6-n})=c_n \delta_p+r^{-n}f$$
with $f=O(r^4)$.

Observe that for $i=0,1,\cdots,[\frac{m}{2}]$,
\begin{equation*}
A_\alpha\big|_{r^{2i}\mathcal{H}_{m-2i}}=(\alpha+2i)(2m-2i+\alpha+n-2)
\end{equation*}
and
$$B_\alpha\big|_{r^{2i}\mathcal{H}_{m-2i}}=2m+2\alpha+n-2,$$
then it yields
\begin{align}\label{triple_A}
& A_{2-n}A_{4-n}A_{6-n}\big|_{r^{2i}\mathcal{H}_{m-2i}}\no\\
=&(6-n+2i)(4-n+2i)(2-n+2i)(2m+4-2i)(2m+2-2i)(2m-2i).
\end{align}

We start to find a formal asymptotic solution like $G_p(x)=r^{6-n}(1+\sum_{k=1}^n \psi_k)+\varphi$ with $\psi_k \in \mathcal{P}_k$. If we can find $\bar \psi=\sum_{k=1}^n \psi_k$ such that
\begin{equation}\label{polynomial_iteration}
A_{2-n}A_{4-n}A_{6-n}\bar \psi+K_{6-n}\bar \psi+f=O(r^{n+1}),
\end{equation}
the regularity theory for elliptic equations gives that there exists a solution $\varphi \in C_{\rm loc}^{6,\alpha}$ for any $0<\alpha <1$ to
\begin{eqnarray*}
P_g(\varphi)=-r^{-n}(A_{2-n}A_{4-n}A_{6-n}\bar \psi+K_{6-n}\bar \psi+f)\in C_{\rm loc}^\alpha.
\end{eqnarray*}
Thus it only remains to seek $\bar \psi$ satisfying \eqref{polynomial_iteration} via induction.
For any nonnegative integer $k$, it is not hard to see from the definition \eqref{K_alpha} for $K_{6-n}$ that $K_{6-n} \varphi \in \mathcal{P}_{k+2}$ when $\varphi \in \mathcal{P}_k$.  We first set $\psi_1=\psi_2=\psi_3=0$ by \eqref{polynomial_iteration} and define
$$f_3=f=O(r^4).$$

\begin{case}\label{odd}
$n$ is odd.
\end{case}
If we have found $\psi_1,\cdots,\psi_k$ for $3 \leq k \leq n-1$ with $\psi_k \in \mathcal{P}_k$ and
$$f_k=A_{2-n}A_{4-n}A_{6-n}\Big(\sum_{i=1}^{k}\psi_i\Big)+K_{6-n}\Big(\sum_{i=1}^{k}\psi_i\Big)+f:=b_{k+1}+O(r^{k+2}),$$
then it follows from \eqref{triple_A} that $A_{2-n}A_{4-n}A_{6-n}$ is invertible on $\mathcal{P}_{k+1}$ for $0 \leq k \leq n-1$. Thus there exists a unique $\psi_{k+1}\in \mathcal{P}_{k+1}$ such that
$$A_{2-n}A_{4-n}A_{6-n}\psi_{k+1}+b_{k+1}=0.$$
This implies that
\begin{align*}
f_{k+1}=&A_{2-n}A_{4-n}A_{6-n}\Big(\sum_{i=1}^{k+1}\psi_i\Big)+K_{6-n}\Big(\sum_{i=1}^{k+1}\psi_i\Big)+f\\
=&f_k+A_{2-n}A_{4-n}A_{6-n}\psi_{k+1}+K_{6-n}\psi_{k+1}\\
=&O(r^{k+2}).
\end{align*}
This finishes the induction and assertion (a) follows.

\begin{case}\label{even_10+}
$n$ is even and not less than $10$.
\end{case}
Since $A_{2-n}A_{4-n}A_{6-n}$ is invertible on $\mathcal{P}_k$ for $0 \leq k \leq n-7$, by the same induction in Case~\ref{odd}, we may find $\psi_1,\cdots, \psi_{n-7}$ such that
\begin{align*}
f_{n-7}=&A_{2-n}A_{4-n}A_{6-n}\Big(\sum_{k=1}^{n-7}\psi_k\Big)+K_{6-n}\Big(\sum_{k=1}^{n-7}\psi_k\Big)+f=O(r^{n-6})\\
:=&b_{n-6}+O(r^{n-5}).
\end{align*}
Let $\psi_{n-6}^{(0)}=\alpha_{n-6}^{(0)}(x)+\beta_{n-6}^{(0)}(x)\log r$, where $\alpha_{n-6}^{(0)}(x) \in \mathcal{P}_{n-6}\setminus r^{n-6}\mathcal{H}_0$ and $\beta_{n-6}^{(0)}(x)\in r^{n-6}\mathcal{H}_0$, then it yields from \eqref{A^3_varphi_logr} that
\begin{align*}
&A_{2-n}A_{4-n}A_{6-n}\psi_{n-6}^{(0)}\\
=&A_{2-n}A_{4-n}A_{6-n}\alpha_{n-6}^{(0)}+(B_{2-n}A_{4-n}A_{6-n}+A_{2-n}B_{4-n}A_{6-n}+A_{2-n}A_{4-n}B_{6-n})\beta_{n-6}^{(0)}.
\end{align*}
Notice that for $0 \leq i \leq \frac{n-8}{2}$, there hold
\begin{eqnarray*}
A_{2-n}A_{4-n}A_{6-n}\big|_{r^{2i}\mathcal{H}_{m-2i}}\neq 0
\end{eqnarray*}
by \eqref{triple_A} and
$$(B_{2-n}A_{4-n}A_{6-n}+A_{2-n}B_{4-n}A_{6-n}+A_{2-n}A_{4-n}B_{6-n})|_{r^{n-6}\mathcal{H}_0}=8(n-2)(n-4)(n-6)\neq 0.$$
Hence there exists a unique $\psi_{n-6}^{(0)}\in \mathcal{P}_{n-6}+\mathcal{P}_{n-6}\log r$ to
$$A_{2-n}A_{4-n}A_{6-n}\psi_{n-6}^{(0)}+b_{n-6}=0.$$
This indicates that
\begin{align*}
f_{n-6}=&f_{n-7}+A_{2-n}A_{4-n}A_{6-n}\psi_{n-6}^{(0)}+K_{6-n}\psi_{n-6}^{(0)}\\
=&O(r^{n-5})+\Big(K_{6-n}\beta_{n-6}^{(0)}\Big) \log r\\
:=&b_{n-5}+O(r^{n-4}) \log r+O(r^{n-4}).
\end{align*}
Let $\psi_{n-5}^{(0)}=\alpha_{n-5}^{(0)}+\beta_{n-5}^{(0)}\log r$, where $\alpha_{n-5}^{(0)} \in \mathcal{P}_{n-5}\setminus r^{n-6}\mathcal{H}_1$ and $\beta_{n-5}^{(0)}\in r^{n-6}\mathcal{H}_1$, then it yields
\begin{align*}
&A_{2-n}A_{4-n}A_{6-n}\psi_{n-5}^{(0)}\\
=&A_{2-n}A_{4-n}A_{6-n}\alpha_{n-5}^{(0)}+(B_{2-n}A_{4-n}A_{6-n}+A_{2-n}B_{4-n}A_{6-n}+A_{2-n}A_{4-n}B_{6-n})\beta_{n-5}^{(0)}.
\end{align*}
By similar arguments, there exists a unique $\psi_{n-5}^{(0)}\in \mathcal{P}_{n-5}+r^{n-6}\mathcal{H}_1\log r$ such that
$$A_{2-n}A_{4-n}A_{6-n}\psi_{n-5}^{(0)}+b_{n-5}=0.$$
This implies that
\begin{align*}
f_{n-5}=&f_{n-6}+A_{2-n}A_{4-n}A_{6-n}\psi_{n-5}^{(0)}+K_{6-n}\psi_{n-5}^{(0)}\\
=&O(r^{n-4})\log r+O(r^{n-4}):=b_{n-4}^{(1)}\log r+O(r^{n-4})+O(r^{n-3})\log r.
\end{align*}
Choose $\psi_{n-4}^{(1)}=\alpha_{n-4}^{(1)}\log r+\beta_{n-4}^{(1)}\log^2 r\in \mathcal{P}_{n-4}\log r+(r^{n-6}\mathcal{H}_2+r^{n-4}\mathcal{H}_0)\log^2 r$, then \eqref{A^3_varphi_logr} gives
\begin{align*}
&A_{2-n}A_{4-n}A_{6-n}\psi_{n-4}^{(1)}\\
=&\Big[A_{2-n}A_{4-n}A_{6-n}\alpha_{n-4}^{(1)}+2(B_{6-n}A_{4-n}A_{2-n}+A_{6-n}B_{4-n}A_{2-n}+A_{6-n}A_{4-n}B_{2-n})\beta_{n-4}^{(1)}\Big]\log r\\
&+A_{2-n}A_{4-n}A_{6-n}\beta_{n-4}^{(1)}\log^2 r+O(r^{n-4}).
\end{align*}
Since
\begin{align*}
&(B_{6-n}A_{4-n}A_{2-n}+A_{6-n}B_{4-n}A_{2-n}+A_{6-n}A_{4-n}B_{2-n})|_{r^{n-6}\mathcal{H}_2}=8(n+2)n(n-2)\neq 0;\\
&(B_{6-n}A_{4-n}A_{2-n}+A_{6-n}B_{4-n}A_{2-n}+A_{6-n}A_{4-n}B_{2-n})|_{r^{n-4}\mathcal{H}_0}=-4n(n-2)(n-4)\neq 0
\end{align*}
and $A_{2-n}A_{4-n}A_{6-n}|_{r^{2i}\mathcal{H}_{n-4-2i}}\neq 0$ for $0 \leq i \leq \frac{n-8}{2}$, then there exists a unique $\psi_{n-4}^{(1)}$ such that
$$A_{2-n}A_{4-n}A_{6-n}\alpha_{n-4}^{(1)}+2(B_{6-n}A_{4-n}A_{2-n}+A_{6-n}B_{4-n}A_{2-n}+A_{6-n}A_{4-n}B_{2-n})\beta_{n-4}^{(1)}+b_{n-4}^{(1)}=0$$
and
\begin{align*}
f_{n-4}^{(1)}=&f_{n-5}+A_{2-n}A_{4-n}A_{6-n}\psi_{n-4}^{(1)}+K_{6-n}\psi_{n-4}^{(1)}\\
=&O(r^{n-4})+O(r^{n-3})\log r+O(r^{n-2})\log^2 r\\
:=&b_{n-4}^{(0)}+O(r^{n-3})\log r+O(r^{n-3})+O(r^{n-2})\log^2 r.
\end{align*}
Choose $\psi_{n-4}^{(0)} \in \mathcal{P}_{n-4}+(r^{n-6}\mathcal{H}_2+r^{n-4}\mathcal{H}_0)\log r$ to remove the term $b_{n-4}^{(0)}$ and set
\begin{align*}
f_{n-4}^{(0)}=&f_{n-4}^{(1)}+A_{2-n}A_{4-n}A_{6-n}\psi_{n-4}^{(0)}+K_{6-n}\psi_{n-4}^{(0)}\\
=&O(r^{n-3})\log r+O(r^{n-3})+O(r^{n-2})\log^2 r.
\end{align*}
By similar arguments and \eqref{A^3_varphi_logr}, we get
\begin{align*}
&\psi_{n-3}^{(1)}\in \mathcal{P}_{n-3}\log r+(r^{n-6}\mathcal{H}_3+r^{n-4}\mathcal{H}_1)\log^2 r;\\
&\psi_{n-3}^{(0)} \in \mathcal{P}_{n-3}+(r^{n-6}\mathcal{H}_3+r^{n-4}\mathcal{H}_1)\log r;\\
&\psi_{n-2}^{(i)}\in\mathcal{P}_{n-2}\log^{i}r+(r^{n-6}\mathcal{H}_4+r^{n-4}\mathcal{H}_2+r^{n-2}\mathcal{H}_0)\log^{i+1}r, \hbox{~~for~~}i=0,1,2;\\
&\psi_{n-1}^{(i)}\in \mathcal{P}_{n-1}\log^{i}r+(r^{n-6}\mathcal{H}_5+r^{n-4}\mathcal{H}_3+r^{n-2}\mathcal{H}_1)\log^{i+1}r, \hbox{~~for~~}i=0,1,2;\\
&\psi_n^{(i)}~~\in \mathcal{P}_n\log^{i}r+(r^{n-6}\mathcal{H}_6+r^{n-4}\mathcal{H}_4+r^{n-2}\mathcal{H}_2)\log^{i+1}r, \hbox{~~for~~}i=0,1,2,3.
\end{align*}
Now we set
$$\psi_{n-6}=\psi_{n-6}^{(0)},\psi_{n-5}=\psi_{n-5}^{(0)},\psi_{n-4}=\psi_{n-4}^{(0)}+\psi_{n-4}^{(1)},\psi_{n-3}=\psi_{n-3}^{(0)}+\psi_{n-3}^{(1)}$$
and
$$
\psi_{n-2}=\sum_{i=0}^2\psi_{n-2}^{(i)},~~\psi_{n-1}=\sum_{i=0}^2\psi_{n-1}^{(i)},\psi_{n}=\sum_{i=0}^3\psi_n^{(i)}.
$$
Eventually, we obtain
\begin{align*}
f_n=&A_{2-n}A_{4-n}A_{6-n}\Big(\sum_{k=1}^n\psi_k\Big)+K_{6-n}\Big(\sum_{k=1}^n\psi_k\Big)+f\\
=&O(r^{n+1})(\log^3 r+\log^2 r+\log r+1)+O(r^{n+2})\log^4 r.
\end{align*}
Hence, $r^{-n}f_n \in C^\alpha$ for any $0<\alpha<1$. This finishes the induction and we obtain assertion (b) as desired.

\begin{case}\label{8D}
$n=8.$
\end{case}
Notice that
$$P_g(G_p(x)-c_nr^{-2})=O(r^{-4})\in L^{p},$$
for some $\frac{8}{5} <p<2$, then it follows from regularity theory of elliptic equations that $G_p(x)-c_nr^{-2} \in C_{\rm loc}^{6-\frac{8}{p}}$. From this, we have
$$G_p(x)=c_n r^{-2}+A+O(r).$$

\begin{case}\label{lcf}
$M$ is locally conformal flat.
\end{case}
One may choose $g$ flat near $p$ and $P_g=-\Delta_0^3$. Hence, one has
$P_g(G(x)-c_n r^{6-n})=0$ and then $G_p(x)-c_n r^{6-n}$ is smooth near $p$.

Therefore, the assertion $(c)$ follows from cases \ref{odd},\ref{8D},\ref{lcf}. In some special cases, the leading term $\psi_4$ can be computed with the help of  Lemma \ref{lem_P_g_r^{6-n}}. The proof of Proposition \ref{prop_expan_G} is complete.

\section{$\mathbf{n \geq10}$ and not locally conformally flat}\label{sec3}

Similar to the Yamabe constant, for $n \geq 3$ and $n \neq 4,6$, we define
$$Y_6^+(M,g)=\inf_{0<u \in H^3(M,g)}\frac{\int_M u P_g u d\mu_g}{\Big(\int_M u^{\frac{2n}{n-6}}d\mu_g\Big)^{\frac{n-6}{n}}}.$$
It follows from \eqref{conf_invar_P_g} that $Y_6^+(M,g)$ is a conformal invariant. However, due to the lack of a maximum principle for higher order elliptic equations in general, we first study another conformally invariant quantity
$$Y_6(M,g)=\inf_{u \in H^3(M,g)\setminus\{0\}}\frac{\int_M u P_g u d\mu_g}{\Big(\int_M |u|^{\frac{2n}{n-6}}d\mu_g\Big)^{\frac{n-6}{n}}}.$$
In particular, we have $Y_6(S^n)=Y_6^+(S^n)=\frac{n-6}{2}Q_{S^n}\omega_n^{\frac{6}{n}}$. For $w \in C_c^\infty(\mathbb{R}^n)$, let
$$\|w\|_{\mathcal{D}^{3,2}}:=\sum\limits_{|\beta|=3}\|D^\beta w\|_{L^2(\mathbb{R}^n)}\approx\|\nabla\Delta w\|_{L^2(\mathbb{R}^n)}$$
and let $\mathcal{D}^{3,2}(\mathbb{R}^n)$ denote the completion of $C_c^\infty(\mathbb{R}^n)$ under this norm. The equivalence of the above last two norms can be easily deduced by the formula \eqref{int_Bochner_formular} below.  We first recall an optimal Euclidean Sobolev inequality (cf.  \cite[p.154-165]{lions}, \cite{lieb}).
\begin{lemma}\label{lem:sharp_Sobolev}
For $n \geq 7$, the following sharp Sobolev embedding inequality holds
$$Y_6(S^n)\Big(\int_{\mathbb{R}^n}|w|^{\frac{2n}{n-6}}dy\Big)^{\frac{n-6}{n}}\leq \int_{\mathbb{R}^n}|\nabla\Delta w|^2 dy
\text{~~for all~~} w \in \mathcal{D}^{3,2}(\mathbb{R}^n).$$
 The equality holds if and only if $w(y)=(\frac{2}
{1+|y|^2})^{\frac{n-6}{2}}$ up to any nonzero constant multiple, as well as all
translations and dilations.
\end{lemma}

\begin{proposition}\label{prop:Y_6(M)_n>9}
On a closed Riemannian manifold $(M^n,g)$ of dimension $n \geq 10$, if there exists $p \in M^n$ such that the Weyl tensor $W_g(p)\neq 0$, then $Y_6(M^n)<Y_6(S^n)$.
\end{proposition}
\begin{proof}
Recall the definition of $P_g$:
$$-P_g=\Delta_g^3+\Delta_g \delta T_2 d+\delta T_2 d \Delta_g+\frac{n-2}{2}\Delta_g(\sigma_1(A)\Delta_g)+\delta T_4 d-\frac{n-6}{2}Q_g,$$
then for all $\varphi\in H^3(M,g)$,
\begin{align*}
\int_M \varphi P_g\varphi d\mu_g =& \int_M|\nabla\Delta\varphi|_g^2d\mu_g-2\int_MT_2(\nabla\varphi,\nabla\Delta\varphi)d\mu_g-\frac{n-2}{2}\int_M\sigma_1(A)(\Delta\varphi)^2d\mu_g  \\
                       & -\int_MT_4(\nabla\varphi,\nabla\varphi)d\mu_g+\frac{n-6}{2}\int_MQ_g\varphi^2d\mu_g.                    
\end{align*}
Fix $\rho>0$ small and choose test functions
$$\varphi(x)=\eta_\rho(x) u_\epsilon(x), \quad u_\epsilon(x)=\Big(\frac{2\epsilon}{\epsilon^2+|x|^2}\Big)^\frac{n-6}{2}, \quad \epsilon>0,$$
where $r=|x|=d_g(x,p)$ and
$$\eta_\rho\in C^\infty_c,~0\leq\eta_\rho\leq1,\quad \eta_\rho\equiv 1\hbox{~~in~~}B_\rho\hbox{~~and~~~}\eta_\rho\equiv 0\hbox{~~in~~}B_{2\rho}^c.$$

It is known from Lee-Park \cite{lp} that up to a conformal factor, under conformal normal coordinates around $p$ with metric $g$, for all $N \geq 5$ there hold
$$\sigma_1(A_g)(p)=0, \; \sigma_1 (A_g)_{,i}(p)=0, \; \Delta_g \sigma_1(A_g)(p)=-\frac{|W(p)|_g^2}{12(n-1)}$$
and
$$\sqrt{\det g}=1+O(r^N).$$

Our purpose is to estimate
$$\int_M \varphi P_g \varphi d\mu_g \hbox{~~and~~} \int_M \varphi^{\frac{2n}{n-6}} d\mu_g.$$
A direct computation shows
\begin{eqnarray*}
u_\epsilon'= -(n-6)u_\epsilon\frac{r}{\epsilon^2+r^2}, \quad u_\epsilon'' = -(n-6)u_\epsilon\frac{\epsilon^2-(n-5)r^2}{(\epsilon^2+r^2)^2}
\end{eqnarray*}
and
\begin{align*}
\Delta_0u_\epsilon =& -(n-6)\frac{u_\epsilon}{(\epsilon^2+r^2)^2}(n\epsilon^2+4r^2), \\
(\Delta_0u_\epsilon)' =& (n-6)(n-4)\frac{u_\epsilon r}{(\epsilon^2+r^2)^3}\Big[(n+2)\epsilon^2+4r^2\Big].
\end{align*}
We start with $\int_M|\nabla\Delta\varphi|_g^2d\mu_g$ and divide its integral into two parts $\int_M=\int_{B_\rho}+\int_{M\setminus \overline{B_\rho}}$. Compute with
\begin{align*}
&\int_{B_\rho}|\nabla\Delta\varphi|_g^2d\mu_g =\int_{B_\rho}g^{ij}(\Delta\varphi)_{,i}(\Delta\varphi)_{,j}d\mu_g \\
&= \int_{B_\rho}(\delta^{ij}+O(r^2))(\Delta_0\varphi+O(r^{N-1})\varphi')_{,i}(\Delta_0\varphi+O(r^{N-1})\varphi')_{,j}(1+O(r^N))dx \\
&= \int_{B_\rho}|(\nabla\Delta)_0\varphi|^2dx+\int_{B_\rho}(\Delta_0\varphi)'(O(r^{N-2})\varphi'+O(r^{N-1})\varphi'')dx
\end{align*}
and
\begin{align*}
\int_{\mathbb{R}^n\setminus \overline{B_\rho}}|(\nabla\Delta)_0\varphi|^2dx
=&(n-6)^2(n-4)^2 \int_{\mathbb{R}^n\setminus \overline{B_\rho}}\frac{u_\epsilon^2r^2}{(\epsilon^2+r^2)^6}\Big[(n+2)\epsilon^2+4r^2\Big]^2dx \\
\leq& C\int_{\rho/\epsilon}^\infty\sigma^{5-n}d\sigma=O\big(\epsilon^{n-6}\big).
\end{align*}
Similarly, we estimate
$$\int_{M\setminus \overline{B_\rho}} |\nabla\Delta\varphi|_g^2 d\mu_g=O(\epsilon^{n-6}).$$
Thus, we obtain
$$\int_M|\nabla\Delta\varphi|_g^2d\mu_g=\int_{\mathbb{R}^n}|\nabla\Delta_0u_\epsilon|^2dx+O\big(\epsilon^{n-6}\big).$$
Secondly, we compute
\begin{align*}
&\int_{B_\rho}\sigma_1(A)(\Delta\varphi)^2d\mu_g\\
=& \int_{B_\rho}(\frac{1}{2}\sigma_1(A)_{,ij}(p)x^ix^j+O(r^3))(\Delta_0\varphi+O(r^{N-1})\varphi')^2(1+O(r^N))dx \\
=& \int_{B_\rho}\frac{1}{2n}\Delta\sigma_1(A)(p)|x|^2(\Delta_0\varphi)^2dx+\int_{B_\rho} O(r^3)\frac{u_\epsilon^2}{(\epsilon^2+r^2)^4}(n\epsilon^2+4r^2)^2dx \\
=& -\frac{(n-6)^2|W(p)|^2}{24n(n-1)}\omega_{n-1}\int_0^\rho\frac{(n\epsilon^2+4r^2)^2}{(\epsilon^2+r^2)^4}u_\epsilon^2 r^{n+1}dr+  \int_{B_\rho}\frac{O(r^3)u_\epsilon^2}{(\epsilon^2+r^2)^2}dx
\end{align*}
and for some large enough $N$
\begin{align*}
\int_{B_{2\rho}\setminus\overline{B_\rho}}\sigma_1(A)(\Delta\varphi)^2d\mu_g
\leq&C\int_{B_{2\rho}\setminus\overline{B_\rho}}|\Delta_0\varphi+O(r^{N-1})\varphi'|^2(1+O(r^N))dx \\
\leq&C\int_{B_{2\rho}\setminus\overline{B_\rho}}\Big[(\Delta_0\varphi)^2+O(r^{2(N-1)})|\varphi'|^2\Big]dx \\
\leq&C\int_{B_{2\rho}\setminus\overline{B_\rho}}\Big(u_\epsilon\Delta_0\eta_\rho+2\nabla u_\epsilon\cdot\nabla\eta_\rho+\eta_\rho\Delta_0 u_\epsilon\Big)^2dx+O(\epsilon^{n-6})\\
\leq&C\int_\rho^{2\rho}\frac{(n\epsilon^2+4r^2)^2}{(\epsilon^2+r^2)^4}u_\epsilon^2r^{n-1}dr+O(\epsilon^{n-6})\\
\stackrel{\sigma=\frac{r}{\epsilon}}{\leq}&C\epsilon^2\int_{\rho/\epsilon}^{2\rho/\epsilon}\frac{(n+4\sigma^2)^2\sigma^{n-1}}{(1+\sigma^2)^{n-2}}d\sigma+O(\epsilon^{n-6})\\
\leq&C\epsilon^2(\frac{\rho}{\epsilon})^{8-n}+O(\epsilon^{n-6})=O(\epsilon^{n-6}).
\end{align*}
Observe that
\begin{equation}\label{hot_est}
\int_{B_\rho}\frac{r^3u_\epsilon^2}{(\epsilon^2+r^2)^2}dx\\
= \left \{\begin{array}{ll}
O(\epsilon^4) &\hbox{~~if~~} n=10,\\
O(\epsilon^5 |\log \epsilon|) &\hbox{~~if~~} n=11,\\
O(\epsilon^5) &\hbox{~~if~~} n\geq 12.
\end{array}
\right.
\end{equation}
Hence,
\begin{align*}
& -\frac{n-2}{2}\int_M\sigma_1(A)(\Delta\varphi)^2d\mu_g \\
=& \frac{(n-6)^2(n-2)|W(p)|^2}{48n(n-1)}\omega_{n-1}\int_0^\rho\frac{(n\epsilon^2+4r^2)^2}{(\epsilon^2+r^2)^4}u_\epsilon^2r^{n+1}dr+\left \{\begin{array}{ll}
O(\epsilon^4) &\hbox{~~if~~} n=10,\\
O(\epsilon^5 |\log \epsilon|) &\hbox{~~if~~} n=11,\\
O(\epsilon^5) &\hbox{~~if~~} n\geq 12.
\end{array}
\right.
\end{align*}
Thirdly, we compute $\int_MT_2(\nabla\varphi,\nabla\Delta\varphi)d\mu_g$.
$$\int_{B_\rho}T_2(\nabla\varphi,\nabla\Delta\varphi)d\mu_g=\int_{B_\rho}\Big[(n-2)\sigma_1(A)\langle\nabla\varphi,\nabla\Delta\varphi\rangle-8A_{ij}\varphi_{,i}(\Delta\varphi)_{,j}\Big]d\mu_g.$$
Observe that $u_{\epsilon,i}=\frac{x^i}{r}u_\epsilon'$ and $(\Delta_0u_\epsilon)_{,i}=\frac{x^i}{r}(\Delta_0u_\epsilon)'$. Then we get
\begin{align*}
&(n-2)\int_{B_\rho}\sigma_1(A)\langle\nabla\varphi,\nabla\Delta\varphi\rangle d\mu_g \\
=& (n-2)\int_{B_\rho}(\frac{1}{2}\sigma_1(A)_{,ij}(p)x^ix^j+O(r^3))g^{kl}\varphi_{,k}(\Delta\varphi)_{,l}d\mu_g \\
=& (n-2)\int_{B_\rho}(\frac{1}{2}\sigma_1(A)_{,ij}(p)x^ix^j+O(r^3))(\delta^{kl}+O(r^2))\varphi_{,k}(\Delta_0\varphi+O(r^{N-1})\varphi')_{,l}d\mu_g \\
=& \frac{n-2}{2}\int_{B_\rho}\frac{1}{n}\Delta\sigma_1(A)(p)|x|^2\varphi_{,i}(\Delta_0\varphi)_{,i}dx+\int_{B_\rho}O(r^3)|\varphi'||(\Delta_0 \varphi)'|dx \\
=& -\frac{(n-2)|W(p)|^2}{24n(n-1)}\int_{B_\rho}\Big\{-(n-6)^2(n-4)\frac{u_\epsilon^2 r^4}{(\epsilon^2+r^2)^4}\Big[(n+2)\epsilon^2+4r^2\Big]\Big\}dx\\
&+\int_{B_\rho}\frac{O(r^3)u_\epsilon^2}{(\epsilon^2+r^2)^2}dx\\
=& \frac{(n-2)(n-4)(n-6)^2}{24n(n-1)}|W(p)|^2\int_{B_\rho}\frac{r^4}{(\epsilon^2+r^2)^4}u_\epsilon^2\Big[(n+2)\epsilon^2+4r^2\Big]dx\\
&+\int_{B_\rho}\frac{O(r^3)u_\epsilon^2}{(\epsilon^2+r^2)^2}dx,
\end{align*}
and
\begin{align*}
& -8\int_{B_\rho}A_{ij}\varphi_{,i}(\Delta\varphi)_{,j}d\mu_g \\
=& -8\int_{B_\rho}\Big(A_{ij,k}(p)x^k+\frac{1}{2}A_{ij,kl}(p)x^kx^l+O(r^3)\Big)\varphi_{,i}(\Delta_0\varphi+O(r^{N-1})\varphi')_{,j}d\mu_g \\
=& -4\int_{B_\rho}A_{ij,kl}(p)x^kx^lx^ix^j\Big[-(n-4)(n-6)^2\frac{u_\epsilon^2}{(\epsilon^2+r^2)^4}[(n+2)\epsilon^2+4r^2]\Big]dx \\
&+\int_{B_\rho}O(r^3)|\varphi'||(\Delta_0 \varphi)'|dx \\
=& 4(n-4)(n-6)^2\int_{B_\rho}\Big[-\frac{2}{9}\frac{1}{n-2}\sum_{k,l}(W_{iklj}(p)x^ix^j)^2-\frac{\sigma_1(A)_{,ij}(p)x^ix^jr^2}{n-2}\Big] \\
& \times\frac{u_\epsilon^2}{(\epsilon^2+r^2)^4}[(n+2)\epsilon^2+4r^2]dx +\int_{B_\rho}\frac{O(r^3)u_\epsilon^2}{(\epsilon^2+r^2)^2}dx\\
=& -\frac{8(n-4)(n-6)^2}{9(n-2)}\int_{B_\rho}\Big[\sum_{k,l}(W_{iklj}(p)W_{sklt}(p)x^ix^jx^sx^t)\frac{u_\epsilon^2}{(\epsilon^2+r^2)^4}[(n+2)\epsilon^2+4r^2]dx \\
& -\frac{4(n-4)(n-6)^2}{n(n-2)}\int_{B_\rho}\frac{\Delta\sigma_1(A)(p)r^4}{(\epsilon^2+r^2)^4}u_\epsilon^2[(n+2)\epsilon^2+4r^2]dx+\int_{B_\rho}\frac{O(r^3)u_\epsilon^2}{(\epsilon^2+r^2)^2}dx \\
=& -\frac{(n-4)(n-6)^2}{(n-1)n(n+2)} \omega_{n-1}|W(p)|^2\int_0^\rho\frac{r^{n+3}u_\epsilon^2}{(\epsilon^2+r^2)^4}[(n+2)\epsilon^2+4r^2]dr\\
&+\int_{B_\rho}\frac{O(r^3)u_\epsilon^2}{(\epsilon^2+r^2)^2}dx,
\end{align*}
where the last identity follows from
\begin{align*}
& \sum_{k,l}W_{iklj}(p)W_{sklt}(p)\int_{B_\rho}x^ix^jx^sx^t\frac{u_\epsilon^2}{(\epsilon^2+r^2)^4}[(n+2)\epsilon^2+4r^2]dx \\
=& \sum_{k,l}W_{iklj}(p)W_{sklt}(p)\int_{\mathbb{S}^{n-1}}\xi^i\xi^j\xi^s\xi^td\mu_{\mathbb{S}^{n-1}}
    \int_0^\rho r^{n+3}\frac{u_\epsilon^2}{(\epsilon^2+r^2)^4}[(n+2)\epsilon^2+4r^2]dr \\
=& \frac{\omega_{n-1}}{n(n+2)}
\sum_{k,l}W_{iklj}(p)W_{sklt}(p)[\delta_{ij}\delta_{st}+\delta_{is}\delta_{jt}+\delta_{it}\delta_{js}] \int_0^\rho \frac{r^{n+3} u_\epsilon^2}{(\epsilon^2+r^2)^4}[(n+2)\epsilon^2+4r^2]dr \\
=& \frac{\omega_{n-1}}{n(n+2)}\Big[|W(p)|^2+W_{iklj}(p)W_{jkli}(p)\Big]\int_0^\rho r^{n+3}\frac{u_\epsilon^2}{(\epsilon^2+r^2)^4}[(n+2)\epsilon^2+4r^2]dr \\
=& \frac{3}{2}\frac{\omega_{n-1}}{n(n+2)}|W(p)|^2\int_0^\rho r^{n+3}\frac{u_\epsilon^2}{(\epsilon^2+r^2)^4}[(n+2)\epsilon^2+4r^2]dr.
\end{align*}
Then we have
\begin{align*}
& -2\int_{B_\rho} T_2(\nabla\varphi,\nabla\Delta\varphi)d\mu_g \\
=& -\frac{(n^2-28)(n-4)(n-6)^2}{12n(n-1)(n+2)}|W(p)|^2\omega_{n-1}\int_0^\rho r^{n+3}\frac{u_\epsilon^2}{(\epsilon^2+r^2)^4}[(n+2)\epsilon^2+4r^2]dr\\
&+\int_{B_\rho}\frac{O(r^3)u_\epsilon^2}{(\epsilon^2+r^2)^2}dx.
\end{align*}
By a similar argument, one has
\begin{align*}
\Big|\int_{B_{2\rho}\setminus \overline{B_\rho}}T_2(\nabla \varphi, \nabla \Delta \varphi)\Big|\leq&C\int_{B_{2\rho}\setminus \overline{B_\rho}}|\nabla \varphi||\nabla \Delta \varphi|d\mu_g\\
\leq&C\int_{B_{2\rho}\setminus \overline{B_\rho}}|u_\epsilon'||(\Delta u_\epsilon)'|dx+O(\epsilon^{n-6})=O(\epsilon^{n-6}).
\end{align*}
Fourthly, we compute $\int_MT_4(\nabla\varphi,\nabla\varphi)d\mu_g$.
\begin{align*}
&(n-6)\int_{B_\rho}\Delta\sigma_1(A)|\nabla\varphi|_g^2d\mu_g\\
=& (n-6)\int_{B_\rho}(\Delta\sigma_1(A)(p)+O(r))(|\varphi'|^2+O(r^2)|\varphi|^2)dx \\
=& -(n-6)^3\frac{|W(p)|^2}{12(n-1)}\int_{B_\rho}\frac{u_\epsilon^2r^2}{(\epsilon^2+r^2)^2}dx+\int_{B_\rho}\frac{O(r^3)u_\epsilon^2}{(\epsilon^2+r^2)^2}dx.
\end{align*}
Using (\ref{quadratic_B}), we get
\begin{align*}
& -\frac{16}{n-4}\int_{B_\rho}B_{ij}\varphi_{,i}\varphi_{,j}d\mu_g=-\frac{16}{n-4}\int_{B_\rho}(n-6)^2u_\epsilon^2\frac{B_{ij}x^ix^j}{(\epsilon^2+r^2)^2}dx \\
=& -\frac{16(n-6)^2}{n-4}\int_{B_\rho}\Big[ -\frac{2}{9}\frac{1}{n-2}\sum_{k,l,s}[(W_{ikls}(p)+W_{ilks}(p))x^i]^2\\
&+\frac{1}{12(n-2)(n-1)}|W(p)|^2r^2-\frac{7n-8}{n-2}\sigma_1(A)_{,ij}(p)x^ix^j+O(r^3)\Big]\frac{u_\epsilon^2}{(\epsilon^2+r^2)^2}dx \\
=& -\tfrac{16(n-6)^2}{n-4}[ -\tfrac{2}{3n(n-2)}+\tfrac{1}{12(n-2)(n-1)}+\tfrac{7n-8}{12(n-2)(n-1)n}]|W(p)|^2\int_{B_\rho}\frac{r^2u_\epsilon^2}{(\epsilon^2+r^2)^2}dx\\
&+\int_{B_\rho}\frac{O(r^3)u_\epsilon^2}{(\epsilon^2+r^2)^2}dx\\
=& \int_{B_\rho}\frac{O(r^3)u_\epsilon^2}{(\epsilon^2+r^2)^2}dx,
\end{align*}
where the second identity follows from
$$\sum_{i,k,l,s}(W_{ikls}(p)+W_{ilks}(p))^2=2|W(p)|^2+2\sum_{i,k,l,s}W_{ikls}(p)W_{ilks}(p)=3|W(p)|^2,$$
in view of
\begin{align*}
0=& W_{ikls}(W_{ilks}+W_{iksl}+W_{islk}) \\
  =& W_{ikls}W_{ilks}+W_{ikls}W_{iksl}+W_{ikls}W_{islk}\\
  =& 2W_{ikls}W_{ilks}-|W|^2 \qquad\hbox{~~at~~} p.
\end{align*}
Also we have
$$\int_{B_{2\rho}\setminus \overline{B_\rho}}T_4(\nabla\varphi,\nabla\varphi)d\mu_g\leq C\int_{B_{2\rho}\setminus \overline{B_\rho}}|\nabla \varphi|_g^2 d\mu_g=O(\epsilon^{n-6}).$$
Hence, collecting the above terms together with \eqref{hot_est}, we obtain
\begin{align*}
&-\int_MT_4(\nabla\varphi,\nabla\varphi)d\mu_g\\
=&-(n-6)\int_{B_\rho}\Delta\sigma_1(A)|\nabla\varphi|_g^2d\mu_g+\frac{16}{n-4}\int_{B_\rho}B_{ij}\varphi_{,i}\varphi_{,j}d\mu_g+O(\epsilon^{n-6})\\
=&(n-6)^3\frac{|W(p)|^2}{12(n-1)}\int_{B_\rho}\frac{u_\epsilon^2r^2}{(\epsilon^2+r^2)^2}dx+\left \{\begin{array}{ll}
O(\epsilon^4) &\hbox{~~if~~} n=10,\\
O(\epsilon^5 |\log \epsilon|) &\hbox{~~if~~} n=11,\\
O(\epsilon^5) &\hbox{~~if~~} n\geq 12.
\end{array}
\right.
\end{align*}
Finally, we compute $\frac{n-6}{2}\int_MQ_g\varphi^2d\mu_g$. By the definition \eqref{Q_g^6} of $Q_g$, integration by parts gives
\begin{align*}
&\frac{n-6}{2}\int_MQ_g\varphi^2d\mu_g\\
=&\frac{n-6}{2}\int_M\Delta^2\sigma_1(A)\varphi^2d\mu_g+\int_{B_\rho}\frac{O(r^3)u_\epsilon^2}{(\epsilon^2+r^2)^2}dx+O(\epsilon^{n-6})\\
=&\frac{n-6}{2}\int_M\Delta\sigma_1(A)\Delta\varphi^2d\mu_g+\int_{B_\rho}\frac{O(r^3)u_\epsilon^2}{(\epsilon^2+r^2)^2}dx+O(\epsilon^{n-6})\\
=& -\tfrac{(n-6)^2|W(p)|^2}{12(n-1)}\omega_{n-1}\int_0^\rho\frac{u_\epsilon^2r^{n-1}}{(\epsilon^2+r^2)^2}[(n-10)r^2-n\epsilon^2]dr+\left \{\begin{array}{ll}
O(\epsilon^4) &\hbox{if~~} n=10,\\
O(\epsilon^5 |\log \epsilon|) &\hbox{if~~} n=11,\\
O(\epsilon^5) &\hbox{if~~} n\geq 12,
\end{array}
\right.
\end{align*}
by \eqref{hot_est}, where the last identity follows from
\begin{align*}
& \frac{n-6}{2}\int_{B_\rho}\Delta\sigma_1(A)\Delta\varphi^2d\mu_g \\
=& \frac{n-6}{2}\int_{B_\rho}(\Delta\sigma_1(A)(p)+O(r))(\Delta_0\varphi^2+O(r^{N-1})(\varphi^2)')dx \\
=& \frac{n-6}{2}\Delta\sigma_1(A)(p)\int_{B_\rho}2(\varphi\Delta_0\varphi+|\nabla\varphi|_0^2)dx+\int_{B_\rho}\frac{O(r)u_\epsilon^2}{\epsilon^2+r^2}dx \\
=& -\frac{(n-6)^2|W(p)|^2}{12(n-1)}\omega_{n-1}\int_0^\rho\frac{u_\epsilon^2r^{n-1}}{(\epsilon^2+r^2)^2}[(n-10)r^2-n\epsilon^2]dr+\int_{B_\rho}\frac{O(r)u_\epsilon^2}{\epsilon^2+r^2}dx
\end{align*}
and the first identity follows from
\begin{eqnarray*}
\Big|\int_{B_{2\rho}\setminus\overline{B_\rho}}Q_g \varphi^2 d\mu_g\Big|\leq C \int_{B_{2\rho}\setminus\overline{B_\rho}} u_\epsilon^2 dx=O(\epsilon^{n-6}).
\end{eqnarray*}

Therefore collecting all the above terms together, we obtain
$$\int_M \varphi P_g \varphi d\mu_g=\int_{\mathbb{R}^n}|\nabla\Delta_0 u_\epsilon|^2 dx+A_{n,\rho,\epsilon}|W(p)|^2 \omega_{n-1}+O(\epsilon^{\min\{n-6,5\}}),$$
where $A_{n,\rho,\epsilon}$ is a constant given by
\begin{align*}
 & (n-6)^2\Big\{\frac{n-2}{48n(n-1)}\int_0^\rho\frac{(n\epsilon^2+4r^2)^2}{(\epsilon^2+r^2)^4}u_\epsilon^2r^{n+1}dr
    +\frac{n-6}{12(n-1)}\int_0^\rho\frac{u_\epsilon^2r^{n+1}}{(\epsilon^2+r^2)^2}dr \\
& -\frac{1}{12(n-1)}\int_0^\rho\frac{u_\epsilon^2r^{n-1}}{(\epsilon^2+r^2)^2}[(n-10)r^2-n\epsilon^2]dr \\
& -\frac{(n^2-28)(n-4)}{12n(n-1)(n+2)}\int_0^\rho r^{n+3}\frac{u_\epsilon^2}{(\epsilon^2+r^2)^4}[(n+2)\epsilon^2+4r^2]dr\Big\} \\
=& 2^{n-6}\frac{(n-6)^2}{12(n-1)}\epsilon^4\Big\{\frac{n-2}{4n}\int_0^{\rho/\epsilon}\frac{(n+4\sigma^2)^2}{(1+\sigma^2)^4}(1+\sigma^2)^{-(n-6)}\sigma^{n+1}d\sigma \\
& +(n-6)\int_0^{\rho/\epsilon}\frac{1}{(1+\sigma^2)^2}(1+\sigma^2)^{-(n-6)}\sigma^{n+1}d\sigma \\
& -\int_0^{\rho/\epsilon}\frac{1}{(1+\sigma^2)^2}(1+\sigma^2)^{-(n-6)}\sigma^{n-1}[(n-10)\sigma^2-n]d\sigma \\
& -\frac{(n^2-28)(n-4)}{n(n+2)}\int_0^{\rho/\epsilon}\frac{\sigma^{n+3}}{(1+\sigma^2)^4}(1+\sigma^2)^{-(n-6)}[(n+2)+4\sigma^2]d\sigma\Big\},
\end{align*}
where $r=\epsilon\sigma$. When $n=10$, we claim that the leading term of the constant in the brace on the right hand side of the above identity:
\begin{align*}
& \frac{1}{5}\int_0^{\rho/\epsilon}\frac{(4\sigma^2+10)^2}{(1+\sigma^2)^4}(1+\sigma^2)^{-4}\sigma^{11}d\sigma
    +\int_0^{\rho/\epsilon}\frac{1}{(1+\sigma^2)^2}(1+\sigma^2)^{-4}(4\sigma^2+10)\sigma^9d\sigma \\
 & -\frac{18}{5}\int_0^{\rho/\epsilon}\frac{1}{(1+\sigma^2)^4}(1+\sigma^2)^{-4}(4\sigma^2+12)\sigma^{13}d\sigma
\end{align*}
is a negative constant multiple of $|\log \epsilon|$. To see this, it is obviously true for the third term and the first two terms equal
\begin{align*}
&\frac{1}{5}\int_0^{\rho/\epsilon}\Big\{\sigma^2[(4\sigma^2+10)^2 -18\sigma^2(4\sigma^2+12)]+5(4\sigma^2+10)(1+\sigma^2)^2\Big\}(1+\sigma^2)^{-8}\sigma^9d\sigma\\
=& \frac{1}{5}\int_0^{\rho/\epsilon}(-36\sigma^6-46\sigma^4+220\sigma^2+50)(1+\sigma^2)^{-8}\sigma^9d\sigma,
\end{align*}
whose leading term is also a negative constant multiple of $|\log \epsilon|$.
For $n\geq11$, let $t=\sigma^2$, the limit of the coefficient of $|W(p)|^2\omega_{n-1}$ as $\epsilon \to 0$ is
\begin{align*}
 & 2^{n-7}\frac{(n-6)^2}{12(n-1)}\epsilon^4\Big\{\frac{n-2}{4n}\int_0^\infty\frac{(n+4t)^2}{(1+t)^{n-2}}t^{\tfrac{n}{2}}dt\no \\
 & +(n-6)\int_0^\infty\frac{1}{(1+t)^{n-4}}t^{\tfrac{n}{2}}dt-\int_0^\infty\frac{(n-10)t-n}{(1+t)^{n-4}}t^{\tfrac{n}{2}-1}dt\no \\
& -\frac{(n^2-28)(n-4)}{n(n+2)}\int_0^\infty\frac{(n+2)+4t}{(1+t)^{n-2}}t^{\tfrac{n}{2}+1}dt\Big\}.
\end{align*}
With the help of Beta function: $\int_0^\infty\frac{x^{\alpha-1}}{(1+x)^{\alpha+\beta}}dx=B(\alpha,\beta)=\frac{\Gamma(\alpha)\Gamma(\beta)}{\Gamma(\alpha+\beta)}$ for ${\rm Re}(\alpha)>0, {\rm Re}(\beta)>0$, we have
\begin{align*}
& \frac{n-2}{4n}\int_0^\infty\frac{(n+4t)^2}{(1+t)^{n-2}}t^{\tfrac{n}{2}}dt \\
=& \frac{n-2}{4n}\int_0^\infty\frac{(n-4)^2+8(n-4)(1+t)+16(1+t)^2}{(1+t)^{n-2}}t^{\tfrac{n}{2}}dt \\
=& \frac{n-2}{4n}\Big[(n-4)^2B(\tfrac{n}{2}+1,\tfrac{n}{2}-3)+8(n-4)B(\tfrac{n}{2}+1,\tfrac{n}{2}-4)+16B(\tfrac{n}{2}+1,\tfrac{n}{2}-5)\Big],
\end{align*}
$$(n-6)\int_0^\infty\frac{1}{(1+t)^{n-4}}t^{\tfrac{n}{2}}dt=(n-6)B(\tfrac{n}{2}+1,\tfrac{n}{2}-5),$$
$$-\int_0^\infty\frac{(n-10)t-n}{(1+t)^{n-4}}t^{\tfrac{n}{2}-1}dt=-(n-10)B(\tfrac{n}{2}+1,\tfrac{n}{2}-5)+nB(\tfrac{n}{2},\tfrac{n}{2}-4),$$
\begin{align*}
& -\frac{(n^2-28)(n-4)}{n(n+2)}\int_0^\infty\frac{(n+2)+4t}{(1+t)^{n-2}}t^{\tfrac{n}{2}+1}dt \\
=& -\frac{(n^2-28)(n-4)}{n(n+2)}\int_0^\infty\frac{4(1+t)^2+(n-6)(1+t)-(n-2)}{(1+t)^{n-2}}t^{\tfrac{n}{2}}dt \\
=& -\frac{4(n^2-28)(n-4)}{n(n+2)}B(\tfrac{n}{2}+1,\tfrac{n}{2}-5)-\frac{(n^2-28)(n-4)(n-6)}{n(n+2)}B(\tfrac{n}{2}+1,\tfrac{n}{2}-4) \\
& +\frac{(n^2-28)(n-4)(n-2)}{n(n+2)}B(\tfrac{n}{2}+1,\tfrac{n}{2}-3).
\end{align*}
Hence, the above limit of the coefficient of $|W(p)|^2\omega_{n-1}$ is rewritten as
\begin{align}\label{coef_|W_g|^2}
& 2^{n-7}\frac{(n-6)^2}{12(n-1)}\epsilon^4\Big\{nB(\tfrac{n}{2},\tfrac{n}{2}-4)\no \\
& +B(\tfrac{n}{2}+1,\tfrac{n}{2}-3)\Big[\frac{n-2}{4n}(n-4)^2+\frac{(n^2-28)(n-4)(n-2)}{n(n+2)}\Big] \no\\
& +B(\tfrac{n}{2}+1,\tfrac{n}{2}-4)\Big[\frac{2(n-2)(n-4)}{n}-\frac{(n^2-28)(n-4)(n-6)}{n(n+2)}\Big] \no\\
& +B(\tfrac{n}{2}+1,\tfrac{n}{2}-5)\Big[\frac{4(n-2)}{n}-n+10+n-6-\frac{4(n^2-28)(n-4)}{n(n+2)}\Big]\Big\} \no\\
=& 2^{n-7}\frac{(n-6)^2}{12(n-1)}B(\tfrac{n}{2}+1,\tfrac{n}{2}-5)\epsilon^4\Big\{(n-10)\no \\
& +\frac{(n-2)(\tfrac{n}{2}-4)(\tfrac{n}{2}-5)}{4n(n+2)(n-3)}(5n^2-2n-120)+\frac{\tfrac{n}{2}-5}{n(n+2)}(-n^3+8n^2+28n-176) \no\\
& +\frac{4}{n(n+2)}(-n^3+6n^2+30n-116)\Big\},
\end{align}
where we have used some elementary identities
\begin{eqnarray*}
B(\tfrac{n}{2}+1,\tfrac{n}{2}-3) = \frac{\Gamma(\tfrac{n}{2}+1)\Gamma(\tfrac{n}{2}-3)}{\Gamma(n-2)}= \frac{(\tfrac{n}{2}-4)(\tfrac{n}{2}-5)}{(n-3)(n-4)}B(\tfrac{n}{2}+1,\tfrac{n}{2}-5),
\end{eqnarray*}
$$B(\tfrac{n}{2}+1,\tfrac{n}{2}-4)=\frac{\tfrac{n}{2}-5}{n-4}B(\tfrac{n}{2}+1,\tfrac{n}{2}-5),$$
\begin{eqnarray*}
B(\tfrac{n}{2},\tfrac{n}{2}-4) = \frac{\Gamma(\tfrac{n}{2})\Gamma(\tfrac{n}{2}-4)}{\Gamma(n-4)}= \frac{n-10}{n}B(\tfrac{n}{2}+1,\tfrac{n}{2}-5).
\end{eqnarray*}
The constant in the last brace of \eqref{coef_|W_g|^2} when $n \geq 11$ is
\begin{align*}
& n-10+\frac{1}{16n(n+2)(n-3)}\Big\{(n-2)(n-8)(n-10)(5n^2-2n-120) \\
& +8(n-3)[(n-10)(-n^3+8n^2+28n-176)+8(-n^3+6n^2+30n-116)]\Big\} \\
=& n-10+\frac{1}{16n(n+2)(n-3)}\Big[-3n^5+2n^4+228n^3-264n^2-1760n-768\Big]\\
=&\frac{-3n^5+18n^4+52n^3-200n^2-800n-768}{16n(n+2)(n-3)}<0.
\end{align*}

On the other hand, we have
$$\int_M \varphi^{\frac{2n}{n-6}}d\mu_g=\int_{B_\rho}u_\epsilon^{\frac{2n}{n-6}}d\mu_g+\int_{B_{2\rho}\setminus\overline{B_\rho}}\varphi^{\frac{2n}{n-6}}d\mu_g=\int_{\mathbb{R}^n}u_\epsilon^{\frac{2n} {n-6}}dx+O(\epsilon^n).$$

Therefore, putting these facts together, we conclude by Lemma \ref{lem:sharp_Sobolev} that
\begin{align*}
&\frac{\int_M \varphi P_g \varphi d\mu_g}{\Big(\int_M \varphi^{\frac{2n}{n-6}}d\mu_g\Big)^{\frac{n-6}{n}}}\\
=&Y_6(S^n)+A_{n,\rho,\epsilon}|W(p)|^2 \omega_{n-1}+\left \{\begin{array}{ll}
O(\epsilon^4) &\hbox{~~if~~} n=10,\\
O(\epsilon^5 |\log \epsilon|) &\hbox{~~if~~} n=11,\\
O(\epsilon^5) &\hbox{~~if~~} n\geq 12,
\end{array}
\right.
\\
=&\left \{\begin{array}{ll}
Y_6(S^n)-C_n|W(p)|^2 \epsilon^4|\log \epsilon|+O(\epsilon^4) &\hbox{~~if~~} n=10,\\
Y_6(S^n)-C_n|W(p)|^2 \epsilon^4+o(\epsilon^4) &\hbox{~~if~~} n \geq 11,
\end{array}
\right.
\end{align*}
for some positive constant $C_n>0$. Consequently, choosing $\epsilon$ sufficiently small, we obtain $Y_6(M^n)<Y_6(S^n)$. This finishes the proof.
\end{proof}

Given a smooth positive function $f$ in $M^n$, we define a ``free" energy functional by
$$E_f[u]=\frac{1}{2}\int_M u P_g u d\mu_g-\frac{1}{2^\sharp}\int_M f |u|^{2^\sharp}d\mu_g.$$
Let $u_{,i}$ or $\nabla_i u$ denote the covariant derivatives of $u$ with respect to metric $g$ and $R_{ijk}^l$ be Riemannian curvature tensor of metric $g$. Notice that
$$\nabla_j \nabla_i\nabla^i  u=\nabla_i\nabla_j\nabla^i u+R_{iij}^k \nabla_k u=\nabla_i\nabla^i\nabla_j u-R_j^k \nabla_k u,$$
there holds
\begin{equation}\label{est_D^3 u}
\int_M |\nabla \Delta u|_g^2 d\mu_g=\int_M |\Delta \nabla_j u-R_j^k \nabla_k u|_g^2 d\mu_g.
\end{equation}
Under $g$-normal coordinates around a point, one gets
\begin{align*}
&\frac{1}{2}\Delta_g |\nabla^2 u|_g^2\\
=&|\nabla^3 u|_g^2+\langle\nabla \Delta \nabla_i u,\nabla \nabla^i u\rangle_g+u_{,ij}(R^l_{ijk}u_{,lk}+R_{j}^l u_{,il}+R_{ijk,k}^lu_{,l}+R_{ijk}^lu_{,lk}).
\end{align*}
Integrating the above identity over $M$ to show
\begin{equation}\label{int_Bochner_formular}
\int_M |\Delta \nabla u|_g^2 d\mu_g=\int_M|\nabla^3 u|_g^2 d\mu_g+\int_M O\big(|{\rm Rm}||\nabla^2 u|_g+|\nabla {\rm Rm}||\nabla u|_g\big)|\nabla^2 u|_gd\mu_g.
\end{equation}
From \eqref{est_D^3 u} and \eqref{int_Bochner_formular}, it yields that the following two norms are equivalent:
\begin{align*}
\|u\|_{H^3}:=&\Big(\int_{M}(|\nabla \Delta u|_g^2 d\mu_g+|\nabla^2 u|_g^2+|\nabla u|_g^2+u^2)d\mu_g\Big)^{\frac{1}{2}}\\
\approx&\Big(\int_{M}(|\nabla^3 u|_g^2 d\mu_g+|\nabla^2 u|_g^2+|\nabla u|_g^2+u^2)d\mu_g\Big)^{\frac{1}{2}}, ~~u \in H^3(M,g).
\end{align*}
Let $\|\cdot\|_p$ denote the norm of $L^p(M,g)$ for $1 \leq p \leq \infty$.

A sequence $\{u_k\}$ in $H^3(M,g)$ is called a Palais-Smale (P-S)$_\beta$ sequence for $E_f$ if $E_f[u_k] \to \beta \in \mathbb{R}$ and $DE_f[u_k] \to 0$ as $k \to \infty$. The energy $E_f$ satisfies (P-S)$_{\beta}$ condition if any Palais-Smale sequence of $E_f$ has a strongly convergent subsequence. We call that $P_g$ is coercive if there exists a constant $\mu(g)>0$ such that
$$\int_M \psi P_g \psi d\mu_g \geq \mu(g)\int_M \psi^2 d\mu_g, \hbox{~~for all~~} \psi \in H^3(M,g).$$

\begin{remark}\label{Einstein_Coercivity_P_g}
If $(M,g)$ is Einstein and of positive constant scalar curvature, from the factorization \eqref{factorization_GJMS} of $P_g$, the coercivity of $P_g$ is automatically satisfied.
\end{remark}

As an application, we adapt some arguments in Esposito-Robert \cite{er} to show some existence results of prescribed $Q$-curvature equation, whose solution may change signs due to  the lack of maximum principles (in general).
\begin{theorem}\label{generalized_main_Thm}
Let $(M^n,g)$ be a smooth closed manifold of dimension $n \geq 10$ and  $f$ be a smooth positive function in $M^n$. Suppose the Weyl tensor $W_g$ is nonzero at a maximum point of $f$ and $f$ satisfies the vanishing order condition \eqref{vanishing_order_con} at this maximum point. Assume $P_g$ is coercive, then there exists a nontrivial $C^{6,\mu} (0<\mu<1)$ solution to
\begin{equation}\label{change_sign_Q-eqn}
P_g u=f|u|^{2^\sharp-2}u \hbox{~~in~~} M.
\end{equation}
In addition, assume $(M,g)$ is  Einstein and of positive scalar curvature, then there exists a smooth solution to the $Q$-curvature equation
\begin{equation}\label{eq:Q-curvature}
P_g u=fu^{\frac{n+6}{n-6}}, u>0\hbox{~~in~~} M.
\end{equation}
\end{theorem}
\begin{proof}
By the assumptions, there exists $p \in M$ such that $f(p)=\max_{x \in M^n}f(x)$,~$W_g(p)\neq 0$ and the vanishing order condition \eqref{vanishing_order_con} of $f$ is true at $p$. Let
$$\gamma_\epsilon(t)=t \frac{\varphi}{\|f^{\frac{1}{2^\sharp}}\varphi\|_{2^\sharp}},$$
where $\varphi=\eta_\rho u_\epsilon$ is the test function chosen in Proposition \ref{prop:Y_6(M)_n>9}. By choosing $t_0$ large enough, we get $E[\gamma_\epsilon(t_0)]<0$. Let
$$\Gamma=\Big\{\gamma(t)\in C([0,t_0],H^3(M,g)); \gamma(0)=0, \gamma(t_0)=t_0 \frac{\varphi}{\|f^{\frac{1}{2^\sharp}}\varphi\|_{2^\sharp}}\Big\}.$$
From the coercivity of $P_g$ and Sobolev embedding theorem, we have
$$E_f[\tfrac{\varphi}{\|f^{\frac{1}{2^\sharp}}\varphi\|_{2^\sharp}}]=\frac{1}{2}\frac{\int_M \varphi P_g \varphi d\mu_g}{\|f^{\frac{1}{2^\sharp}}\varphi\|_{2^\sharp}^2}-\frac{1}{2^\sharp}\geq \frac{1}{2}C-\frac{1}{2^\sharp}.$$
It only suffices to estimate the term:
\begin{align*}
\int_M f \varphi^{\frac{2n}{n-6}}d\mu_g=&\int_{B_\rho}\Big[f(p)+\sum_{k=2}^4\frac{1}{k!}\partial_{i_1\cdots i_k}f(p)x^{i_1}\cdots x^{i_k}+O(|x|^5)\Big]u_\epsilon^{2^\sharp}dx+O(\epsilon^n)\\
=&f(p)\int_{\mathbb{R}^n}u_\epsilon^{\frac{2n}{n-6}}dx+\left \{\begin{array}{ll}
O(\epsilon^4) &\hbox{~~if~~} n=10,\\
o(\epsilon^4) &\hbox{~~if~~} n \geq 11,
\end{array}
\right.
\end{align*}
where the second equality follows from the vanishing order condition \eqref{vanishing_order_con} of $f$ at $p$. From this and some existing estimates in the proof of Proposition \ref{prop:Y_6(M)_n>9}, we conclude that there exist some sufficiently small $\epsilon>0$ and a constant $C'_n>0$ such that
\begin{align*}
\sup_{t \geq 0}E_f[\gamma_\epsilon(t)]=&E_f[\gamma_\epsilon(t^\ast)]=\frac{3}{n}\Big(\frac{\int_M \varphi P_g \varphi d\mu_g}{\|f^{\frac{1}{2^\sharp}}\varphi\|_{2^\sharp}^2}\Big)^{\frac{2^\sharp}{2^\sharp-2}}\\
\leq&\left \{\begin{array}{ll}
\frac{3}{n}(\max_{M}f)^{\frac{6-n}{6}}Y_6(S^n)^{\frac{n}{6}}-C'_n|W(p)|^2 \epsilon^4|\log \epsilon|+O(\epsilon^4) &\hbox{~~if~~} n=10,\\
\frac{3}{n}(\max_{M}f)^{\frac{6-n}{6}}Y_6(S^n)^{\frac{n}{6}}-C'_n|W(p)|^2 \epsilon^4+o(\epsilon^4) &\hbox{~~if~~} n \geq 11,
\end{array}
\right.
\end{align*}
where $t^\ast=\Big(\frac{\int_M \varphi P_g \varphi d\mu_g}{\|f^{\frac{1}{2^\sharp}}\varphi\|_{2^\sharp}^2}\Big)^{\frac{1}{2^\sharp-2}}$.
Then it follows from Mountain Pass Lemma (cf. \cite{ar} or  \cite[ Proposition 1]{er}) that
$$\beta=\inf\limits_{\gamma \in \Gamma}\sup\limits_{0 \leq t \leq t_0} E_f[\gamma(t)] \leq \sup\limits_{t \geq 0} E_f[\gamma_\epsilon(t)]<\frac{3}{n}Y_6(S^n)^{\frac{n}{6}}(\max_{M}f)^{\frac{6-n}{6}}$$
is a critical value of $E_f$ and there exists a (P-S)$_\beta$ sequence $\{u_k\}$ of $E_f$ in $H^3(M,g)$.

Next we claim that $E_f$ satisfies (P-S)$_\beta$ condition. For the above (P-S)$_\beta$ sequence $\{u_k\}$ satisfying $E_f[u_k] \to \beta$ and $DE_f[u_k] \to 0$ as $k \to \infty$, there holds
\begin{eqnarray*}
2\beta+o(\|u_k\|_{H^3})=2E_f[u_k]-\langle DE_f[u_k],u_k\rangle=\frac{6}{n}\int_M f|u_k|^{2^\sharp} d\mu_g.
\end{eqnarray*}
Together with the coercivity of $P_g$, one has
$$\mu(g)\|u_k\|_{H^3} \leq 2E_f[u_k]+\frac{2}{2^\sharp}\int_M f|u_k|^{2^\sharp} d\mu_g\leq C+o(\|u_k\|_{H^3}).$$
From this, we get $\{u_k\}$ is bounded in $H^3(M,g)$. Then up to a subsequence, as $k \to \infty$, $u_k\rightharpoonup u$ in $H^3(M,g)$ and $u_k \to u$ in $L^p(M,g)$ for $1 \leq p < 2^\sharp$. It is easy to verify that $u$ is a weak solution to \eqref{change_sign_Q-eqn}, that is, for all $\psi \in H^3(M,g)$,
$$\int_M \psi P_g ud\mu_g=\int_M f|u|^{2^\sharp-2}u \psi d\mu_g.$$
Choosing $\psi=u$, one has
$$\int_{M}uP_g ud\mu_g=\int_M f|u|^{2^\sharp}d\mu_g,$$
whence
$$E_f[u]=\frac{3}{n}\int_M f|u|^{2^\sharp} d\mu_g \geq 0.$$
Applying Brezis-Lieb lemma to
\begin{align*}
\int_M|\nabla \Delta u_k|_g^2 d\mu_g=&\int_M |\nabla \Delta u|_g^2 d\mu_g+\int_M |\nabla \Delta (u-u_k)|_g^2 d\mu_g+o(1),\\
\int_M f|u_k|^{2^\sharp} d\mu_g=&\int_M f|u|^{2^\sharp} d\mu_g+\int_M f|u-u_k|^{2^\sharp} d\mu_g+o(1),
\end{align*}
we have
\begin{align*}
E_f[u_k]-E_f[u]=&\frac{1}{2}\int_M|\nabla \Delta (u-u_k)|_g^2-\frac{1}{2^\sharp}\int_M f|u-u_k|^{2^\sharp} d\mu_g+o(1)\\
=&E_f[u-u_k]+o(1).
\end{align*}
Since $DE_f[u_k] \to 0$ in $(H^3(M,g))'$, we have
\begin{align*}
o(1)=&\langle u_k-u,DE_f[u_k]\rangle=\langle u_k-u, DE_f[u_k]-DE_f[u]\rangle\\
=&\int_M |\nabla \Delta (u-u_k)|_g^2d\mu_g-\int_M f|u-u_k|^{2^\sharp}d\mu_g+o(1).
\end{align*}
Thus, we obtain
$$\frac{3}{n}\int_M |\nabla \Delta (u-u_k)|_g^2 d\mu_g+o(1)=E_f[u_k-u]=E_f[u_k]-E_f[u]+o(1) \leq E_f[u_k]+o(1) \to \beta,$$
as $k \to \infty$, which yields
\begin{equation}\label{eq:(u_k-u)_H^3}
\int_M |\nabla \Delta (u-u_k)|_g^2 d\mu_g\leq \frac{n}{3}\beta+o(1).
\end{equation}
Mimicking a cut-and-paste argument as in \cite{dhl}, we obtain that given $\epsilon>0$, there exists a constant $B_\epsilon>0$ such that
$$\Big(\int_M |\psi|^{2^\sharp}d\mu_g\Big)^{\frac{2}{2^\sharp}}\leq (1+\epsilon)Y_6(S^n)^{-1}\int_M (|\nabla \Delta \psi|_g^2+|\nabla^2 \psi|_g^2+|\nabla \psi|_g^2)d\mu_g+B_\epsilon\int_M \psi^2 d\mu_g,$$
for all $\psi \in H^3(M,g)$. Choosing $\psi=u_k-u$ and $k$ sufficiently large to get
$$\Big(\int_M |u-u_k|^{2^\sharp} d\mu_g\Big)^{\frac{2}{2^\sharp}}\leq (1+\epsilon)Y_6(S^n)^{-1}\int_M |\nabla \Delta (u-u_k)|_g^2d\mu_g+o(1).$$
Hence we have
\begin{align*}
o(1)=&\int_M |\nabla \Delta (u-u_k)|_g^2d\mu_g-\int_M f|u-u_k|^{2^\sharp}d\mu_g\\
\geq& \int_M |\nabla \Delta (u-u_k)|_g^2d\mu_g\Big[1-(\max_M f)(1+\epsilon)^{\frac{2^\sharp}{2}}Y_6(S^n)^{-\frac{2^\sharp} {2}}\big(\int_M |\nabla \Delta (u-u_k)|_g^2d\mu_g\big)^{\frac{6}{n-6}}\Big].
\end{align*}
From \eqref{eq:(u_k-u)_H^3} and $\beta<\frac{3}{n}Y_6(S^n)^{\frac{n}{6}}(\max_{M}f)^{\frac{6-n}{6}}$, choosing $\epsilon$ sufficiently small, we get
$$o(1) \geq C \int_M|\nabla \Delta (u-u_k)|_g^2d\mu_g.$$
Combining the above inequality and the coercivity of $P_g$ to show that $u_k \to u$ in $H^3(M,g)$. Using the regularity result in Lemma \ref{lem:regularity} below, we know that $u \in C^{6,\mu}(M)$ for any $0<\mu<1$.

In addition, assume $(M,g)$ is Einstein and has positive constant scalar curvature, we define the modified energy in $H^3(M,g)$ by
$$E_f^+[u]=\frac{1}{2}\int_M uP_g u d\mu_g-\frac{1}{2^\sharp}\int_M fu_+^{2^\sharp}d\mu_g,$$
where $u_+=\max\{u,0\}$. Using the above similar arguments associated with Mountain Pass Lemma and mimicking what we did in Lemma \ref{lem:regularity} below for $E_f^+$, we get that there exists a nontrivial $C^6$-solution $u$ to
\begin{equation}\label{eq:u_change_sign}
P_g u=fu_+^{\frac{n+6}{n-6}}  \hbox{~~in~~} M.
\end{equation}
 Since $P_g$ is coercive by Remark \ref{Einstein_Coercivity_P_g}, testing equation \eqref{eq:u_change_sign} with $u_-=\min\{u,0\}$ we conclude that $u \geq 0$ in $M$. Together with $R_g$ being a positive constant and the factorization \eqref{factorization_GJMS} of GJMS operator:
$$\Big(-\Delta_g+\frac{(n-6)(n+4)}{4n(n-1)}R_g\Big)\Big(-\Delta_g+\frac{(n-4)(n+2)}{4n(n-1)}R_g\Big)\Big(-\Delta_g+\frac{n-2}{4(n-1)}R_g\Big)u\geq 0$$
and $u \not \equiv 0$ in $M$, we employ the maximum principle twice and strong maximum principle once for elliptic equations of second order to show that $u$ is a positive solution to the equation \eqref{eq:Q-curvature}.  From this and Schauder estimates for elliptic equations, we conclude that $u \in C^\infty(M)$. This completes the proof.
\end{proof}

We are now concerned with the regularity of mountain pass critical points for $E$.

\begin{lemma}\label{lem:regularity}
Let $(M,g)$ be a smooth closed Riemannian manifold of dimension $n\geq7$. Assume $u\in H^3(M,g)$ is a weak solution of equation \eqref{change_sign_Q-eqn}. Then $u \in C^{6,\mu}(M)$ for any $0< \mu<1$.
\end{lemma}
\begin{proof}
Rewrite $P_g=(-\Delta_g)^3-M_g+\frac{n-6}{2}Q_g$ by \eqref{P_g}. Let $u\in H^3(M,g)$ be a weak solution of equation \eqref{change_sign_Q-eqn} and rewrite this equation as
\begin{align}
(-\Delta_g+1)^3u =& M_g u+3\Delta_g^2 u-3\Delta_g u+(1-\tfrac{n-6}{2}Q_g)u+f|u|^{2^\sharp-2}u \no\\
                    :=& b+f|u|^{2^\sharp-2}u, \label{eq:regular_rewrite}
\end{align}
where $b\in H^{-1}(M,g)$. By the Sobolev embedding theorem we have $u\in L^{2^\sharp}(M,g)$ and $|u|^{2^\sharp-2}\in L^{\frac{n}{6}}(M,g)$. Given $\epsilon>0$, there exist a $K_\epsilon>0$ and a decomposition of $f|u|^{2^\sharp-2}=h_\epsilon+\eta_\epsilon$ with $\|h_\epsilon\|_{\frac{n}{6}}\leq\epsilon$ , $\|\eta_\epsilon\|_\infty\leq K_\epsilon$. Inspired by the arguments in \cite[Proposition 3]{er}, for $s>1$ we define an operator
$$H_\epsilon:v\in L^s(M,g)\rightarrow(-\Delta_g+1)^{-3}(h_\epsilon v)\in L^s(M,g).$$
Indeed, from Sobolev embedding theorem, the standard $W^{2,p}$-regularity theory of the elliptic operator $-\Delta_g+1$ and H\"older's inequality, we have
\begin{align*}
    \|H_\epsilon v\|_s \leq& C\|(-\Delta_g+1)^{-3}(h_\epsilon v)\|_{W^{6,\frac{ns}{n+6s}}}\leq C\|h_\epsilon v\|_{\frac{ns}{n+6s}}\\
        \leq& C\|h_\epsilon\|_{\frac{n}{6}}\|v\|_s\leq C\epsilon\|v\|_s,
\end{align*}
where the constant $C$ is independent of $u$. Choose $\epsilon>0$ small enough, then the norm of $H_\epsilon$ on space $L^s(M,g)$ satisfies
$$\|H_\epsilon\|_{L^s\rightarrow L^s}\leq C\epsilon \leq \frac{1}{2}.$$
With the help of the operator $H_\epsilon$, we rewrite equation \eqref{eq:regular_rewrite} as
$$(Id-H_\epsilon)u=(-\Delta_g+1)^{-3}(b+\eta_\epsilon u),$$
then it is easy to show $Id-H_\epsilon:L^s\rightarrow L^s$ is bounded and invertible. We intend to prove $u\in H^6(M,g)$. To see this, notice that $(-\Delta_g+1)^{-3}(b+\eta_\varepsilon u)\in H^5(M,g)$ since $b+\eta_\varepsilon u\in H^{-1}(M,g)$. In the following, we first show $u\in H^4(M,g)$.
Applying the Sobolev embedding theorem and the $L^s$ boundedness of the operator $(Id-H_\varepsilon)^{-1}$ to show that if $n\leq10$, $u\in L^p(M,g)$ for all $p>1$,
and if $n>10$, $u\in L^{\frac{2n}{n-10}}(M,g)$. In the latter case we have $|u|^{2^\sharp-2}u\in L^{\frac{2n(n-6)}{(n+6)(n-10)}}(M,g)$. From equation \eqref{eq:regular_rewrite},we get
$$(-\Delta_g+1)^2u=(-\Delta_g+1)^{-1}b+(-\Delta_g+1)^{-1}(f|u|^{2^\sharp-2}u).$$
From $(-\Delta_g+1)^{-1}(|u|^{2^\sharp-2}u)\in W^{2,\frac{2n(n-6)}{(n+6)(n-10)}}(M,g)\hookrightarrow L^2(M,g)$ and $(-\Delta_g+1)^{-1}b\in L^2(M,g)$, we have $u\in H^4(M,g)$ in both cases.
Repeat the above step with $u\in H^4(M,g)$ and  $b\in L^2(M,g)$ in this situation. Notice that $(-\Delta_g+1)^{-3}(b+\eta_\epsilon u)\in H^6(M,g)$, similar arguments in the above step show that if $n\leq12$, $u\in L^p(M,g)$ for all $p>1$ and if $n>12$, $u\in L^{\frac{2n}{n-12}}(M,g)$.
In the latter case, we get $|u|^{2^\sharp-2}u\in L^2(M,g)$ due to $\frac{2n(n-6)}{(n+6)(n-12)}>2$. Hence we obtain $u\in H^6(M,g)$.

Finally we start with the classical bootstrap. We now construct a non-decreasing sequence $s_k\in\mathbb{R}\cup\{+\infty\}$ such that $u\in W^{6,s_k}(M,g)$ for all $k\in\mathbb{N}$.
Set $s_0=2$, and find $k\geq0$ such that $u\in W^{6,s_k}(M,g)$. Next we will define $s_{k+1}$ by induction. The Sobolev embedding theorem yields
$$b\in L^{\frac{ns_k}{n-2s_k}}(M,g),$$
with the convention that $\frac{ns_k}{n-2s_k}=+\infty$ if $s_k\geq\frac{n}{2}$, and
$$|u|^{2^\sharp-2}u\in L^{\frac{ns_k(n-6)}{(n-6s_k)(n+6)}}(M,g),$$
with the convention that $\frac{ns_k}{n-6s_k}=+\infty$ if $s_k\geq\frac{n}{6}$. In view of equation \eqref{eq:regular_rewrite}, we have $u\in W^{6,s_{k+1}}(M,g)$ with $s_{k+1}=\min\{\frac{ns_k}{n-2s_k},\frac{ns_k(n-6)}{(n-6s_k)(n+6)}\}$. If $s_k \in \mathbb{R}$ for all $k \in \mathbb{N}$, it must hold that $s_k\rightarrow+\infty$. Then we have $u\in W^{6,p}(M,g)$ for all $1\leq p<+\infty$. If $s_k=+\infty$ for all $k \geq k_0+1$, then $s_{k_0}\geq\frac{n}{6}$,
whence $b\in L^{\frac{n}{4}}(M,g)$ and $|u|^{2^\sharp-2}u\in L^q(M,g)$ for all $1\leq q<+\infty$. The equation \eqref{eq:regular_rewrite} leads to $u\in W^{6, \frac{n}{4}}(M)$. Repeating the argument twice, we obtain $u\in W^{6,p}(M,g)$ for all $1\leq p<+\infty$.
From this and the Sobolev embedding theorem, we have $u\in C^{5,\nu}(M)$ for all $0<\nu<1$.
By the regularity theory for the classical solution of the elliptic operator $-\Delta_g+1$, we get $u\in C^{6,\mu}(M)$ for some $0<\mu<1$. This completes the proof.
\end{proof}

\appendix
\section{Appendix: Proof of Lemma \ref{lem_P_g_r^{6-n}}}\label{app_A}

As in Proposition \ref{prop:Y_6(M)_n>9}, one may employ all computations under conformal normal coordinates of metric $g$ around a point in $M$. From Lee-Park \cite{lp} that up to a conformal factor, under $g$-conformal normal coordinates around this point, for all $N \geq 5$ there hold
$$\sigma_1(A_g)=0, \; \sigma_1 (A_g)_{,i}=0, \; \Delta_g \sigma_1(A_g)=-\frac{|W|_g^2}{12(n-1)}$$
at this point and 
$$\sqrt{\det g}=1+O(r^N) \hbox{~~near this point}.$$

To simplify the notation, we will omit the subscript $g$. Notice that
\begin{align*}
-P_g(r^{6-n})=&\Big[\Delta^3+\Delta \delta T_2 d+\delta T_2 d \Delta+\frac{n-2}{2}\Delta(\sigma_1(A)\Delta)+\delta T_4 d-\frac{n-6}{2}Q_g\Big](r^{6-n})\\
:=&\sum_{k=1}^6 I_k.
\end{align*}
Next, we begin to estimate all terms $I_1$-$I_6$.

For $I_1$, let $u=u(r)$ be a radial function, we have
\begin{align*}
       \Delta u(r)   =& \Delta_0 u(r)+O(r^{N-1})u'; \\
       \Delta^2 u(r) =& \Delta_0(\Delta_0 u(r)+O(r^{N-1})u')+O(r^{N-1})(\Delta_0 u(r)+O(r^{N-1})u')' \\
                       =& \Delta^2_0 u(r)+O(r^{N-1})u'''+O(r^{N-2})u''+O(r^{N-3})u'; \\
       \Delta^3 u(r) =& \Delta^3_0 u(r)+O(r^{N-1})u^{(5)}+O(r^{N-2})u^{(4)}+O(r^{N-3})u''' \\
                       & +O(r^{N-4})u''+O(r^{N-5})u'.
\end{align*}
Hence we obtain
\begin{eqnarray*}
       I_1=\Delta^3(r^{6-n})=-c_n\delta_p+O(r^{N-n}).
\end{eqnarray*}

To estimate $I_2$. Notice that
\begin{align*}
I_2=&\Delta\delta T_2d(r^{6-n}) \\
=& -\Delta[(T_2)_{ij}(r^{6-n})_{,j}]_{,i} \\
                           =& -\Delta[(T_2)_{ij,i}(r^{6-n})_{,j}+(T_2)_{ij}(r^{6-n})_{,ji}].
\end{align*}
Using
\begin{align}\label{Hess_radial_fnc}
    (r^{6-n})_{,j}  =& (6-n)r^{4-n}x^j, \no\\
    (r^{6-n})_{,ji} =& (4-n)(6-n)r^{2-n}x^ix^j+(6-n)r^{4-n}\delta_{ij}+O(r^{6-n}),
\end{align}
one has
\begin{align*}
(T_2)_{ij,i}(r^{6-n})_{,j} =& (n-10)\sigma_1(A)_{,j}(6-n)r^{4-n}x^j \\
                           =& (n-10)(6-n)\sigma_1(A)_{,j}x^j r^{4-n}
\end{align*}
and
\begin{align*}
    & (T_2)_{ij}(r^{6-n})_{,ji} \\
    =& [(n-2)\sigma_1(A)g_{ij}-8A_{ij}](6-n)[(4-n)r^{2-n}x^ix^j+r^{4-n}\delta_{ij}+O(r^{6-n})] \\
    =& (6-n)[4(n-4)\sigma_1(A)r^{4-n}-8(4-n)A_{ij}x^ix^jr^{2-n}]+O(r^{7-n}).
\end{align*}
Hence, we obtain
\begin{align*}
    & I_2 \\
    =& -(6-n)\Delta[(n-10)\sigma_1(A)_{,j}x^j r^{4-n}+4(n-4)\sigma_1(A)r^{4-n}-8(4-n)A_{ij}x^ix^jr^{2-n}] \\
    =& (n-6)\{(n-10)[4(4-n)\sigma_1(A)_{,j}x^jr^{2-n}+2(4-n)\sigma_1(A)_{,jk}x^jx^kr^{2-n} \\
    & +\sigma_1(A)_{,jkk}x^jr^{2-n}+2\Delta\sigma_1(A)r^{4-n}] +O(r^{5-n})\\
    & +4(n-4)[\Delta\sigma_1(A)r^{4-n}+2(4-n)\sigma_1(A)_{,k}x^kr^{2-n}+2(4-n)\sigma_1(A)r^{2-n}] \\
    & +8(n-4)[4(2-n)A_{ij}x^ix^jr^{-n}+\Delta A_{ij}x^ix^jr^{2-n}+4\sigma_1(A)_{,i}x^ir^{2-n}+2\sigma_1(A)r^{2-n}]\} \\
    =& (n-6)\Big\{-4(n-4)(3n-26)\sigma_1(A)_{,j}x^jr^{2-n}+6(n-6)\Delta\sigma_1(A)r^{4-n} \\
     & -2(n-10)(n-4)\sigma_1(A)_{,jk}x^jx^kr^{2-n}+(n-10)\sigma_1(A)_{,jkk}x^jr^{4-n}+O(r^{5-n})\\
     & -8(n-6)(n-4)\sigma_1(A)r^{2-n}-32(n-4)(n-2)A_{ij}x^ix^jr^{-n}+8(n-4)\Delta A_{ij}x^ix^jr^{2-n}\Big\} \\
    =& (n-6)\Big\{-4(n-4)(3n-26)\sigma_1(A)_{,ij}(p)x^ix^jr^{2-n} -\frac{n-6}{2(n-1)}|W(p)|^2r^{4-n} \\
     & -2(n-10)(n-4)\sigma_1(A)_{,ij}(p)x^ix^jr^{2-n}-4(n-6)(n-4)\sigma_1(A)_{,ij}(p)x^ix^jr^{2-n}\\
     & -16(n-4)(n-2)A_{ij,kl}(p)x^ix^jx^kx^lr^{-n}+8(n-4)\Delta A_{ij}x^ix^jr^{2-n}\Big\}+O(r^{5-n}) \\
    =& (n-6)\Big\{-2(n-4)(9n-74)\sigma_1(A)_{,ij}(p)x^ix^jr^{2-n} -\frac{n-6}{2(n-1)}|W(p)|^2r^{4-n} \\
    & -16(n-4)(n-2)A_{ij,kl}(p)x^ix^jx^kx^lr^{-n}+8(n-4)\Delta A_{ij}x^ix^jr^{2-n}\Big\}+O(r^{5-n}).
\end{align*}

To estimate
\begin{align*}
I_3=&\delta T_2d\Delta(r^{6-n})\\
=& -[(T_2)_{ij}(\Delta r^{6-n})_{,j}]_{,i} \\
=& -(T_2)_{ij,i}(\Delta r^{6-n})_{,j}-(T_2)_{ij}(\Delta r^{6-n})_{,ji}.
\end{align*}
Recall that $T_2=(n-2)\sigma_1(A)g-8A,$ then
$$(T_2)_{ij,i}=(n-10)\sigma_1(A)_{,j}.$$
Observe that
\begin{eqnarray*}
    \Delta r^{6-n}= 4(6-n)r^{4-n}+O(r^{N+4-n}),
\end{eqnarray*}
$$(\Delta r^{6-n})_{,j}=4(6-n)(4-n)x^jr^{2-n}+O(r^{N+3-n}),$$
and
\begin{eqnarray*}
(\Delta r^{6-n})_{,ji}= 4(6-n)(4-n)[(2-n)x^ix^jr^{-n}+r^{2-n}\delta_{ij}]+O(r^{4-n}),
\end{eqnarray*}
then it yields
\begin{align*}
    & (T_2)_{ij}(\Delta r^{6-n})_{,ji} \\
    =& 4(n-6)(n-4)[(n-2)\sigma_1(A)g_{ij}-8A_{ij}][(2-n)x^ix^jr^{-n}+r^{2-n}\delta_{ij}+O(r^{4-n})] \\
    =& 4(n-6)(n-4)[-(n-2)^2\sigma_1(A)r^{2-n}+n(n-2)\sigma_1(A)r^{2-n} \\
    & +8(n-2)r^{-n}A_{ij}x^ix^j-8\sigma_1(A)r^{2-n}]+O(r^{5-n})\\
    =& 4(n-6)(n-4)\Big[2(n-6)\sigma_1(A)r^{2-n}+8(n-2)r^{-n}A_{ij}x^ix^j\Big]+O(r^{5-n})\\
    =& 4(n-6)(n-4)[(n-6)\sigma_1(A)_{,ij}(p)x^ix^j r^{2-n}+4(n-2)r^{-n}(A_{ij,kl}(p)x^ix^jx^kx^l)]+O(r^{5-n}).
\end{align*}
Hence, we obtain
\begin{align*}
I_3=& -4(n-6)(n-4)\Big[(n-6)\sigma_1(A)_{,ij}(p)x^ix^j r^{2-n}+4(n-2)r^{-n}(A_{ij,kl}(p)x^ix^jx^kx^l)\Big]\\
&-4(n-6)(n-4)(n-10)r^{2-n}\sigma_1(A)_{,i}x^i+O(r^{5-n}) \\
    =& -8(n-8)(n-6)(n-4)\sigma_1(A)_{,ij}(p)x^ix^j r^{2-n} \\
    & -16(n-6)(n-4)(n-2)r^{-n}(A_{ij,kl}(p)x^ix^jx^kx^l)+O(r^{5-n}).
\end{align*}

We now compute
\begin{align*}
    I_4=& \frac{n-2}{2}\Delta(\sigma_1(A)\Delta(r^{6-n})) \\
    =& 2(n-2)(6-n)\Delta(\sigma_1(A)r^{4-n})+O(r^{N+4-n}) \\
    =& 2(n-2)(6-n)r^{2-n}[\Delta\sigma_1(A)r^2+2(4-n)\sigma_1(A)_{,i}x^i+2(4-n)\sigma_1(A)]+O(r^{N+2-n}) \\
    =& 2(n-2)(n-6)r^{2-n}\Big[\frac{1}{12(n-1)}|W(p)|^2r^2+3(n-4)\sigma_1(A)_{,ij}(p)x^i x^j\Big]+O(r^{5-n}).
\end{align*}

For $I_5$, from \eqref{Hess_radial_fnc} we have
\begin{align*}
        I_5=&\delta T_4d(r^{6-n})\\
        =& -((T_4)_{ij}{r^{6-n}}_{,j})_{,i} \\
                            =& -(T_4)_{ij,i}(r^{6-n})_{,j}-(T_4)_{ij}(r^{6-n})_{,ji}\\
                            =& (n-6)\big[r^{4-n}(T_4)_{ij,i}x^j-(n-4)r^{2-n}(T_4)_{ij}x^ix^j+r^{4-n}{\rm tr}(T_4)\big]\\
                            :=& (n-6)[I_1^{(5)}+I_2^{(5)}+I_3^{(5)}].
\end{align*}
Also from \cite{lp}, there holds
$${\rm Sym}\Big(R_{kl,ij}+\frac{2}{9}R_{nklm}R_{nijm}\Big)(p)=0 \hbox{~~and~~} R_{ij}(p)=0,$$
then
$$R_{kl,ij}(p)x^ix^jx^kx^l=-\frac{2}{9}W_{nklm}(p)W_{nijm}(p)x^ix^jx^kx^l.$$
Thus we have
\begin{eqnarray}\label{Akl,ij}
A_{kl,ij}(p)x^ix^jx^kx^l = -\frac{2}{9}\frac{1}{n-2}\sum_{k,l}(W_{iklj}(p)x^ix^j)^2-\frac{\sigma_1(A)_{,ij}(p)x^ix^jr^2}{n-2}.
\end{eqnarray}
To estimate $I_3^{(5)}$. From the definition of $T_4$, one gets
\begin{align*}
       {\rm tr}(T_4)=& -\tfrac{3n^3-12n^2-36n+64}{4}\sigma_1(A)^2+4(n^2-4n-12)|A|^2+n(n-6)\Delta\sigma_1(A)\\
                     =& -\frac{n(n-6)}{12(n-1)}|W(p)|^2+O(r).
\end{align*}
Thus one obtains
$$I_3^{(5)}=-\frac{n(n-6)}{12(n-1)}|W(p)|^2 r^{4-n}+O(r^{5-n}).$$
For the term $I_1^{(5)}$, it is easy to see
\begin{eqnarray*}
  I_1^{(5)}=r^{4-n} (T_4)_{ij,i} x^j=O(r^{5-n}).
\end{eqnarray*}
It remains to estimate the term $I_2^{(5)}$, one has
\begin{equation}\label{quadratic_T_4}
    (T_4)_{ij}x^ix^j = (n-6)\Delta\sigma_1(A)r^2-\frac{16}{n-4}B_{ij}x^ix^j+O(r^4).
\end{equation}
Notice that
\begin{align*}
    B_{ij}x^ix^j =& [C_{ijk,k}-A_{kl}W_{kijl}]x^ix^j=[(A_{ij,k}-A_{ik,j})_{,k}-A_{kl}W_{kijl}]x^ix^j \\
                 =& [\Delta A_{ij}-A_{ik,jk}+O(r)]x^ix^j
\end{align*}
and
\begin{align*}
    \Delta(A_{ij}x^ix^j) =& (A_{ij,k}x^ix^j+A_{ij}(x^i\delta_{jk}+x^j\delta_{ik}))_{,k} \\
                         =& (\Delta A_{ij})x^ix^j+2A_{ij,k}(x^i\delta_{jk}+x^j\delta_{ik})+2\sigma_1(A)+O(r^3) \\
                         =& (\Delta A_{ij})x^ix^j+4\sigma_1(A)_{,i}x^i+2\sigma_1(A)+O(r^3).
\end{align*}
By \eqref{Akl,ij}, one gets
\begin{align}\label{quadratic_Delta_A}
    (\Delta A_{ij})x^ix^j=& \Delta(A_{ij}x^ix^j)-4\sigma_1(A)_{,i}x^i-2\sigma_1(A)+O(r^3)\no\\
                          =& \Delta\Big[\frac{1}{2}A_{ij,kl}(p)x^ix^jx^kx^l+O(r^5)\Big]-4[\sigma_1(A)_{,ij}(p)x^ix^j +O(r^3)] \no\\
                          & -\sigma_1(A)_{,ij}(p)x^ix^j+O(r^3)\no\\
                          =&  \Delta\Big[-\frac{1}{9}\frac{1}{n-2}\sum_{k,l}(W_{iklj}(p)x^ix^j)^2-\frac{\sigma_1(A)_{,ij}(p)x^ix^jr^2}{2(n-2)}\Big]\no\\
                          & -5\sigma_1(A)_{,ij}(p)x^ix^j+O(r^3)\no\\
                          =& -\frac{2}{9}\frac{1}{n-2}\sum_{k,l,s}[(W_{ikls}(p)+W_{ilks}(p))x^i]^2+\frac{1}{12(n-2)(n-1)}|W(p)|^2r^2\no\\
                          &-6\frac{n-1}{n-2}\sigma_1(A)_{,ij}(p)x^ix^j+O(r^3),
\end{align}
where the last identity follows from the following two estimates:
\begin{align*}
    \Delta(\sigma_1(A)_{,ij}(p)x^ix^jr^2)
    =& \Delta(\sigma_1(A)_{,ij}(p)x^ix^j)r^2+2\nabla_s(\sigma_1(A)_{,ij}(p)x^ix^j)\nabla_sr^2 \\
    & +(\sigma_1(A)_{,ij}(p)x^ix^j)\Delta r^2\\
    =& 2\Delta\sigma_1(A)(p)r^2+8\sigma_1(A)_{,ij}(p)x^ix^j+2n\sigma_1(A)_{,ij}(p)x^ix^j+O(r^3) \\
    =& -\frac{1}{6(n-1)}|W(p)|^2r^2+2(n+4)\sigma_1(A)_{,ij}(p)x^ix^j+O(r^3)
\end{align*}
and
$$\Delta\sum_{k,l}(W_{iklj}(p)x^ix^j)^2 =2\sum_{k,l,s}[W_{iklj}(p)(x^i\delta_{js}+x^j\delta_{is})]^2
                                      =2\sum_{k,l,s}[(W_{ikls}(p)+W_{ilks}(p))x^i]^2,$$
which follows from
$$\Delta\Big[\sum_{k,l}(W_{iklj}(p)x^ix^j)^2\Big]=2\sum_{k,l}[(W_{iklj}(p)x^ix^j)\Delta(W_{sklt}(p)x^sx^t)+|\nabla(W_{iklj}(p)x^ix^j)|^2]$$
and $\Delta(W_{sklt}(p)x^sx^t)=(W_{sklt}(p)(x^s\delta_{it}+x^t\delta_{is}))_{,i}=2W_{sklt}(p)\delta_{st}=0$.
Using $A_{ik,jk}=A_{ik,kj}+R^l_{ijk}A_{lk}+R^l_{kjk}A_{il}=\sigma_1(A)_{,ij}+R_{lijk}A_{lk}+R_{lj}A_{il}$, one has
\begin{align*}
    & A_{ik,jk}x^ix^j \\
    =& \sigma_1(A)_{,ij}x^ix^j+R_{lijk}A_{lk}x^ix^j+R_{lj}A_{il}x^ix^j \\
    =& (\sigma_1(A)_{,ij}(p)+O(r))x^ix^j+(W_{lijk}(p)+O(r))(A_{lk,m}(p)x^m+O(r^2))x^ix^j+O(r^4) \\
    =& \sigma_1(A)_{,ij}(p)x^ix^j+O(r^3).
\end{align*}
Thus, one obtains
\begin{align}
    B_{ij}x^ix^j =& -\frac{2}{9}\frac{1}{n-2}\sum_{k,l,s}[(W_{ikls}(p)+W_{ilks}(p))x^i]^2+\frac{1}{12(n-2)(n-1)}|W(p)|^2r^2\no\\
    & -\frac{7n-8}{n-2}\sigma_1(A)_{,ij}(p)x^ix^j+O(r^3).\label{quadratic_B}
\end{align}
Inserting the above equations into \eqref{quadratic_T_4}, one gets
\begin{align*}
    & (T_4)_{ij}x^ix^j \\
    =& -\frac{n-6}{12(n-1)}|W(p)|^2r^2+\frac{32}{9(n-4)(n-2)}\sum_{k,l,s}[(W_{ikls}(p)+W_{ilks}(p))x^i]^2 \\
   & -\frac{4}{3(n-4)(n-2)(n-1)}|W(p)|^2r^2+\frac{16(7n-8)}{(n-4)(n-2)}\sigma_1(A)_{,ij}(p)x^ix^j+O(r^3),
\end{align*}
whence
\begin{align*}
I_2^{(5)}=&r^{2-n}\Big[\frac{(n-6)(n-4)}{12(n-1)}|W(p)|^2r^2-\frac{32}{9(n-2)}\sum_{k,l,s}\big((W_{ikls}(p)+W_{ilks}(p))x^i\big)^2 \\
& +\frac{4}{3(n-2)(n-1)}|W(p)|^2r^2-\frac{16(7n-8)}{n-2}\sigma_1(A)_{,ij}(p)x^ix^j\Big]+O(r^{5-n}).
\end{align*}
Combining all the terms together, one has
\begin{align*}
   I_5 =& \Big[-\frac{n^2-8n+8}{3(n-1)(n-2)}|W(p)|^2r^{4-n}-\frac{32}{9(n-2)}\sum_{k,l,s}\big((W_{ikls}(p)+W_{ilks}(p))x^i\big)^2 r^{2-n} \\
    &~~ -\frac{16(7n-8)}{n-2}\sigma_1(A)_{,ij}(p)x^ix^j r^{2-n}\Big](n-6)+O(r^{5-n}).
\end{align*}

Finally, from the definition of $Q_g$ in \eqref{Q_g^6}, it is not hard to show that
$$I_6=-\frac{n-6}{2}Q_g r^{6-n}=O(r^{6-n}).$$

Therefore, collecting all the terms $I_1$-$I_6$ together with (\ref{Akl,ij}) and (\ref{quadratic_Delta_A}), we conclude that
\begin{align*}
&-P_g(r^{6-n}) \\
=& -c_n \delta_p+(n-6)\Big[-\frac{16}{9}\sum_{k,l,s}\big((W_{ikls}(p)+W_{ilks}(p))x^i\big)^2 r^{2-n}-\frac{2(n-8)}{3(n-1)}|W(p)|^2r^{4-n}\\
&+\tfrac{64(n-4)}{9}\sum_{k,l}(W_{iklj}(p)x^ix^j)^2 r^{-n}-4(5n^2-66n+224)\sigma_1(A)_{,ij}(p)x^ix^jr^{2-n}\Big]+O(r^{5-n})\\
                =&-c_n \delta_p+(n-6)r^{-n}\Big\{\tfrac{64(n-4)}{9}\Big[\sum_{k,l}(W_{iklj}(p)x^ix^j)^2-\tfrac{r^2}{n+4}\sum_{k,l,s}\big((W_{ikls}(p)+W_{ilks}(p))x^i\big)^2\\
                &+\tfrac{3}{2(n+4)(n+2)}|W(p)|^2 r^4\Big]+\tfrac{16(3n-20)}{9(n+4)}r^2\Big[\sum_{k,l,s}\big((W_{ikls}(p)+W_{ilks}(p))x^i\big)^2-\tfrac{3}{n}|W(p)|^2 r^2\Big]\\
                &-4(5n^2-66n+224)r^2\Big[\sigma_1(A)_{,ij}(p)x^ix^j+\frac{|W(p)|^2}{12n(n-1)}r^2\Big]\\
                &+\frac{3n^4-16n^3-164n^2+400n+2432}{3(n+4)(n+2)n(n-1)}|W(p)|^2 r^4\Big\}+O(r^{5-n}),
\end{align*}
where each term in square brackets on the right hand side of the last identity is harmonic polynomial. This finishes the proof of Lemma \ref{lem_P_g_r^{6-n}}.


\begin{thebibliography}{99}
%
%
\bibitem{ar}
A. Ambrosetti and P. Rabinowitz,\textit{ Dual variational methods in critical point theory and applications}, J. Functional Analysis 14 (1973), 349-381.

\bibitem{aubin}
T. Aubin, \textit{ $\acute{E}$quations diff$\acute{e}$rentielles non lin$\acute{e}$aires et probl$\grave{e}$me de Yamabe concernant la courbure scalaire}, J. Math. Pures Appl. (9) 55 (1976), no. 3, 269-296.

\bibitem{dhl}
Z. Djadli, E. Hebey and M. Ledoux, \textit{ Paneitz-type operators and applications}, Duke Math. J. 104 (2000), no. 1, 129-169.

\bibitem{er}
P. Esposito and F. Robert, \textit{ Mountain pass critical points for Paneitz-Branson operators}, Calc. Var. Partial Differential Equations 15 (2002), no. 4, 493-517.

\bibitem{fg}
C. Fefferman and Graham, \textit{ The Ambient metric}, Annals of Mathematics Studies, 178, Princeton University Press, Princeton, NJ, 2012. x+113 pp.

\bibitem{gjms}
C. Graham, R. Jenne, L. Mason and G. Sparling,\textit{ Conformally invariant powers of the Laplacian. I. Existence}, J. London Math. Soc. (2) 46 (1992), no. 3, 557-565.

\bibitem{gover}
A. Gover,\textit{ Laplacian operators and Q-curvature on conformally Einstein manifolds},  Math. Ann. 336 (2006), no. 2, 311-334.

\bibitem{gursky_hang_lin}
M. Gursky, F. Hang and Y. Lin, \textit{ Riemannian manifolds with positive Yamabe invariant and Paneitz operator}, preprint (2015), arXiv:1502.01050.

\bibitem{gur_mal}
M. Gursky and A. Malchiodi, \textit{A strong maximum principle for the Paneitz operator and a non-local flow for the $Q$-curvature,} to appear in JEMS, arXiv:1401.3216.

\bibitem{hy3}
F. Hang and P. Yang, \textit{ $Q$ curvature on a class of manifolds with dimension at least five,} preprint (2014), arXiv:1411.3926.

\bibitem{juhl}
A. Juhl, \textit{Explicit formulas for GJMS-operators and Q-curvatures}, Geom. Funct. Anal. 23 (2013), no. 4, 1278-1370.

\bibitem{lp}
J. Lee and T. Parker, \textit{ The Yamabe problem}, Bull. Amer. Math. Soc. (N.S.) 17 (1987), no. 1, 37-91.

\bibitem{lx}
Y. Li and J. Xiong, \textit{ Compactness of conformal metrics with constant $Q$-curvature. I}, preprint (2015), arXiv:1506.00739.

\bibitem{lieb}
E. Lieb,\textit{ Sharp constants in the Hardy-Littlewood-Sobolev and related inequalities}, Ann. of Math.
(2) 118 (1983), 349-374.

\bibitem{lions}
P. L. Lions, \textit{The concentration-compactness principle in the
calculus of variations: The limit case. Part I,}  Rev. Mat.
Iberoamericana 1 (1985), 145-201.

\bibitem{stein}
E. Stein, \textit{Singular integrals and differentiability properties of functions}, Princeton University Press, Princeton, New Jersey, 1970.

\bibitem{wunsch}
V. W\"unsch, \textit{ On conformally invariant differential operators}, Math. Nachr. 129 (1986), 269-281.

\end{thebibliography}
\end{document}